\documentclass[12pt]{amsart}

\usepackage{amssymb}
\usepackage{graphicx}
\usepackage{stmaryrd}
\usepackage{color}
\usepackage{pifont}
\usepackage{enumerate}

\newcommand{\ee}{\varepsilon}

\newcommand{\N}{{\mathbb N}}

\newcommand{\R}{{\mathbb R}}

\newtheorem{thm}{Theorem}[section]

\newtheorem{lem}[thm]{Lemma}
\newtheorem{cor}[thm]{Corollary}

\newtheorem{prop}[thm]{Proposition}
\newtheorem{ex}[thm]{Example}

\newtheorem{prob}[thm]{Problem}

\newcommand{\relint}{{\mathrm{relint}}\,}
\newcommand{\cl}{{\mathrm{cl}}\,}
\newcommand{\conv}{{\mathrm{conv}}\,}

\newcommand{\lin}{{\mathrm{lin}}\,}

\newcommand{\di}{\diamondsuit}

\begin{document}
\hfill\today
\bigskip

\title{Symmetrization in geometry}
\author[Gabriele Bianchi, Richard J. Gardner, and Paolo Gronchi]
{Gabriele Bianchi, Richard J. Gardner, and Paolo Gronchi}
\address{Dipartimento di Matematica e Informatica ``U. Dini", Universit\`a di Firenze, Viale Morgagni 67/A, Firenze, Italy I-50134} \email{gabriele.bianchi@unifi.it}
\address{Department of Mathematics, Western Washington University,
Bellingham, WA 98225-9063, United States} \email{richard.gardner@wwu.edu}
\address{Dipartimento di Matematica e Informatica ``U. Dini", Universit\`a di Firenze, Piazza Ghiberti 27, Firenze, Italy I-50122} \email{paolo.gronchi@unifi.it}
\thanks{First and third author supported in part by the Gruppo
Nazionale per l'Analisi Matematica, la Probabilit\`a e le loro
Applicazioni (GNAMPA) of the Istituto Nazionale di Alta Matematica (INdAM).  Second author supported in
part by U.S.~National Science Foundation Grant DMS-1402929.}
\subjclass[2010]{Primary: 52A20, 52A39; secondary: 28B20, 52A38, 52A40} \keywords{convex body, Steiner symmetrization, Schwarz symmetrization, Minkowski symmetrization, central symmetrization}


\begin{abstract}
The concept of an $i$-symmetrization is introduced, which provides a convenient framework for most of the familiar symmetrization processes on convex sets.  Various properties of $i$-symmetrizations are introduced and the relations between them investigated. New expressions are provided for the Steiner and Minkowski symmetrals of a compact convex set which exhibit a dual relationship between them.  Characterizations of Steiner, Minkowski and central symmetrization, in terms of natural properties that they enjoy, are given and examples are provided to show that none of the assumptions made can be dropped or significantly weakened. Other familiar symmetrizations, such as Schwarz symmetrization, are discussed and several new ones introduced.
\end{abstract}

\maketitle

\section{Introduction}

Around 1836, Jakob Steiner introduced the process now known as Steiner symmetrization in attempting to prove the isoperimetric inequality.  His proof was incomplete, since he assumed the existence of the extremum, but a standard modern approach (see, for example, \cite[Chapter~9]{Gru07}) is still based on Steiner symmetrization.  Indeed, Steiner symmetrization remains an extremely potent technique in geometry, where it has found frequent use, for instance in the demonstration of a variety of powerful affine isoperimetric inequalities.  See, for example, \cite[Chapter~9]{Gar06}, \cite[Chapter~9]{Gru07}, \cite[Chapter~10]{Sch93}, and the references given there.  Beyond geometry, Steiner symmetrization plays an important role in several areas of mathematics, particularly analysis and PDEs.  The latter development was stimulated by the appearance of the classic text of P\'{o}lya and Szeg\"{o} \cite{PS}, which inspired a huge number of works.  See, for example,
\cite{B80, F00, Hay, Hen, Kaw, K06, KP, LL, Str}, and the references given in these texts.

Despite the vast literature surrounding Steiner symmetrization and its applications, we are not aware of a characterization of it, and one purpose of this paper is to provide some.  We also formulate a general framework for many symmetrizations: For $i\in \{0,\dots,n-1\}$ and an $i$-dimensional subspace $H$ in $\R^n$, we call a map $\di$, from a class $\mathcal{B}$ of nonempty compact sets in $\R^n$ to the subclass $\mathcal{B}_H$ of members of $\mathcal{B}$ that are $H$-symmetric (i.e., symmetric with respect to $H$), an {\em $i$-symmetrization} on $\mathcal{B}$.  With this terminology in place, we show that {\em Steiner symmetrization is the unique $(n-1)$-symmetrization on convex bodies in $\R^n$, $n\ge 2$, that is monotonic, volume preserving, and either invariant on $H$-symmetric spherical cylinders or projection invariant.}  (See Section~\ref{symm} for the definitions of these properties and Section~\ref{subsec:notations} for basic terminology and notation.)  The version assuming invariance on $H$-symmetric spherical cylinders is a consequence of a result we prove for Steiner symmetrization on compact sets in $\R^n$, $n\ge 2$.  Examples are given that suggest that the familiar generalization of Steiner symmetrization called Schwarz symmetrization may be difficult to classify in a nontrivial manner.

Another process familiar in geometry is now usually called Minkowski symmetrization, despite being introduced by Blaschke (see \cite[p.~174]{Gru07} and \cite[p.~181]{Sch93}), because up to a scaling factor it involves taking the Minkowski sum of a set and its reflection in a subspace.  The significance of the Minkowski symmetral of a compact convex set stems partly from the fact that it contains the Steiner symmetral of the set.  This relationship has been found particularly useful in studying the convergence of successive Steiner symmetrals.  See, for example, \cite{CouD14}, \cite[Chapter~9]{Gru07}, \cite[Notes for Sections~3.3 and 10.3]{Sch93}, and the references there for the many deep results on this topic by various authors.  We prove that {\em Minkowski symmetrization is the unique $(n-1)$-symmetrization on convex bodies (or on compact convex sets) in $\R^n$, $n\ge 2$, that is monotonic, mean width preserving, and either invariant on $H$-symmetric spherical cylinders or projection invariant.}

The paper is structured as follows. In Section~\ref{symm}, as well as introducing $i$-symmetrizations, we define the eight main properties of them that we find the most useful and indicate which of them are or are not enjoyed by the main known symmetrizations, namely, $p$th central symmetrization and Steiner, Schwarz, Minkowski, Minkowski-Blaschke, and Blaschke symmetrizations.  To this list we add another, fiber symmetrization, which includes Steiner and Minkowski symmetrization as special cases. We regard this as having essentially been introduced by McMullen \cite{McM}, though we substantially generalize the concept.  In Theorem~\ref{Blassym} we prove that when $n\geq3$ and $i\ge 1$, Blaschke symmetrization is not projection invariant.

Section~\ref{MS}, on projection covariant symmetrizations, may be regarded as a sequel to the investigation of Gardner, Hug, and Weil \cite{GHW, GHW2} into additions, such as Minkowski and $L_p$ addition, in convex geometry.  In fact, certain symmetrizations, such as central symmetrization, result from adding a set to its reflection in the origin.  Such procedures form a subclass of the $0$-symmetrizations (i.e., $i$-symmetrizations with $i=0$), and \cite[Section~8]{GHW} contains several results classifying members of this subclass.  In particular, \cite[Corollary~8.4]{GHW} (see also \cite{AG}) classifies central symmetrization, defined by (\ref{csdeff}) below.  Another new symmetrization, $M$-symmetrization, is introduced in Section~\ref{MS} and employs the notion of $M$-addition studied in \cite{GHW, GHW2}. There are many other symmetrization processes in geometry, such as those leading to the fundamental notions of projection body, intersection body, and centroid body (see \cite{Gar06, Sch93}).  These are examples of $0$-symmetrizations not covered by the results in \cite{GHW} or the present paper.  However, characterizations of these and related bodies have been obtained using valuation theory; see, for example \cite{L1, L3}.

Two more natural generalizations of Steiner and Minkowski symmetrization, that we call the inner and outer rotational symmetrizations, are defined in Section~\ref{further}, along with several others that are useful in showing that the properties we assume in our results cannot be omitted.

Section~\ref{relations} examines how the various properties of $i$-symmetrizations relate to each other.  A significant role is played by monotonicity and idempotence, the natural property that repeating the symmetrization with respect to the same subspace has no effect.  In Theorem~\ref{Newchar}, we also obtain new expressions for the Steiner and Minkowski symmetrals of a compact convex set that bring to light the dual relationship between them. The latter of these expressions is used in obtaining our characterization of Minkowski symmetrization mentioned above. Corollary~\ref{IdemGab} gives properties which ensure that an $i$-symmetral of a compact convex set contains its fiber symmetral and is contained in its Minkowski symmetral.  We also find properties which ensure that an $i$-symmetral of a compact convex set contains its inner rotational symmetral and is contained in its outer rotational symmetral; see Theorem~\ref{IdemGab2}.  Such results lead to others concerning the convergence of successive $i$-symmetrals.  For example, Corollary~\ref{Successivecor} gives sufficient conditions which guarantee that successive symmetrizations of a convex body will converge to a ball. No attempt is made in the short Section~\ref{Convergence} to obtain the best results of this type; the topic will be thoroughly investigated in a future paper.

Two characterizations of Minkowski symmetrization are given in Section~\ref{MinS}, the highlight being Theorem~\ref{MinChar}(i), the one mentioned earlier.  This follows from Theorem~\ref{SchwarzF}, which provides conditions (different to those in Corollary~\ref{IdemGab}) under which the $(n-1)$-symmetral of a compact convex set is contained in its Minkowski symmetral.  In Theorem~\ref{MinChar}(iii), we also obtain a new characterization of the central symmetral.

Section~\ref{SS} focuses on Steiner symmetrization, both on compact sets and on compact convex sets or convex bodies.  Among other results, the characterization of Steiner symmetrization on convex bodies referred to above may be found in Corollary~\ref{Steiner}.

Throughout the paper, we attempt to provide examples which show that none of our assumptions can be dropped or significantly weakened and succeed in this endeavor when $i=n-1$ (in particular, for the main results mentioned above) and with just three exceptions otherwise.  The final Section~\ref{problems} lists the corresponding open problems.  Here too we pose the intriguing Problem~\ref{prob0}, part of which asks whether there is a symmetrization on compact convex sets, which like Minkowski and Steiner symmetrization is monotonic and either invariant on $H$-symmetric spherical cylinders or projection invariant, but which preserves surface area instead of mean width or volume.

We thank a referee for a very thorough reading that led to many improvements.

\section{Preliminaries}\label{subsec:notations}

As usual, $S^{n-1}$ denotes the unit sphere and $o$ the origin in Euclidean $n$-space $\R^n$.  We assume throughout that $n\ge 2$.   The standard orthonormal basis for $\R^n$ is $\{e_1,\dots,e_n\}$.  The unit ball in $\R^n$ will be denoted by $B^n$. If $x,y\in \R^n$ we write $x\cdot y$ for the inner product and $[x,y]$ for the line segment with endpoints $x$ and $y$. If $x\in \R^n\setminus\{o\}$, then $x^{\perp}$ is the $(n-1)$-dimensional subspace orthogonal to $x$. Throughout the paper, the term {\em subspace} means a linear subspace.

If $X$ is a set,  we denote by $\lin X$, $\conv X$, $\cl X$, $\relint X$, and $\dim X$ the {\it linear hull}, {\it convex hull}, {\it closure}, {\it relative interior}, and {\it dimension} (that is, the dimension of the affine hull) of $X$, respectively.  If $H$ is a subspace of $\R^n$, then $X|H$ is the (orthogonal) projection of $X$ on $H$ and $x|H$ is the projection of a vector $x\in \R^n$ on $H$.

If $X$ and $Y$ are sets in $\R^n$ and $t\ge 0$, then $tX=\{tx:x\in X\}$ and
$$X+Y=\{x+y: x\in X, y\in Y\}$$
denotes the {\em Minkowski sum} of $X$ and $Y$.

When $H$ is a fixed subspace of $\R^n$, we use $X^{\dagger}$ for the {\em reflection} of $X$ in $H$, i.e., the image of $X$ under the map that takes $x\in \R^n$ to $2(x|H)-x$.  If $X^{\dagger}=X$, we say $X$ is {\em $H$-symmetric}. If $H=\{o\}$, we instead write $-X=(-1)X$ for the reflection of $X$ in the origin and {\it $o$-symmetric} for $\{o\}$-symmetric. A set $X$ is called {\em rotationally symmetric} with respect to $H$ if for all $x\in H$, $X\cap (H^{\perp}+x)=r_x(B^n\cap H^{\perp})+x$ for some $r_x\ge 0$.  If $\dim H=n-1$, then a compact convex set is rotationally symmetric with respect to $H$ if and only if it is $H$-symmetric.  The term {\em $H$-symmetric spherical cylinder} will always mean a set of the form $D_r(x)+s(B^n\cap H^{\perp})=D_r(x)\times s(B^n\cap H^{\perp})$, where $s>0$ and $D_r(x)\subset H$ is the ball with $\dim D=\dim H$, center $x$, and radius $r>0$.  Of course, $H$-symmetric spherical cylinders are rotationally symmetric with respect to both $H$ and $H^{\perp}$.

The phrase {\em translate orthogonal to $H$} means translate by a vector in $H^{\perp}$.

We write ${\mathcal H}^k$ for $k$-dimensional Hausdorff measure in $\R^n$, where $k\in\{1,\dots, n\}$. The notation $dz$  means integration with respect to ${\mathcal H}^k$ for the appropriate $k$.

The Grassmannian of $k$-dimensional subspaces in $\R^n$ is denoted by ${\mathcal{G}}(n,k)$.

We denote by ${\mathcal C}^n$ the class of nonempty compact subsets of $\R^n$.
Let ${\mathcal K}^n$ be the class of nonempty compact convex
subsets of $\R^n$ and let ${\mathcal K}^n_n$ be the class of {\em convex bodies}, i.e., members of ${\mathcal K}^n$ with interior points.  A subscript $s$ denotes the $o$-symmetric sets in these classes.  If $K\in {\mathcal K}^n$, then
$$
h_K(x)=\sup\{x\cdot y: y\in K\},
$$
for $x\in\R^n$, defines the {\it support function} $h_K$ of $K$.  The texts by Gr\"{u}ber \cite{Gru07} and Schneider \cite{Sch93} contain a wealth of useful information about convex sets and related concepts such as the {\em intrinsic volumes} $V_j$, $j\in\{1,\dots, n\}$ (see also \cite[Appendix~A]{Gar06}).  In particular, if $K\in {\mathcal K}^n$, then $V_1(K)$ and $V_{n-1}(K)$ are (up to constants independent of $K$) the {\em mean width} and {\em surface area} of $K$, respectively. If
$\dim K=k$, then $V_k(K)={\mathcal H}^k(K)$ and in this case we prefer to write $V_k(K)$.  By $\kappa_n$ we denote the volume $V_n(B^n)$ of the unit ball in $\R^n$.

If $M$ is an arbitrary subset of $\R^2$, we define the {\em $M$-sum} $K\oplus_M L$ of arbitrary sets $K$ and $L$ in $\R^n$ by
\begin{equation}\label{Mdef}
K\oplus_M L= \{ a x + b y : x\in K, y\in L, (a,b )\in M\}.
\end{equation}
See \cite{GHW} and \cite{Mes} for more information and historical remarks concerning $M$-addition.

Let $H\in {\mathcal{G}}(n,i)$, $i\in\{0,\dots,n\}$.  If $p\in \R^n$, write $p=(x,y)$, where $x\in H$ and $y\in H^{\perp}$ satisfy $p=x+y$.  Suppose that $s,t\in \R$ and $K,L\in {\mathcal K}^n$.  The {\em fiber combination} $(s\circ K)\nplus_H (t\circ L)$ of $K$ and $L$ relative to $H$, defined by
\begin{equation}\label{Fiberdef}
(s\circ K)\nplus_H (t\circ L)= \{(x,sy+tz) : (x,y)\in K, (x,z)\in L\},
\end{equation}
was introduced by McMullen \cite{McM}, who noted that $(s\circ K)\nplus_H (t\circ L)\in {\mathcal K}^n$,  $(s\circ K)\nplus_H (t\circ L)=sK+tL$ if $i=0$, and $K\nplus_H L=K\cap L$ if $i=n$.  (We have adapted the definition in \cite{McM} to suit our purposes.)

\section{$i$-Symmetrization: Properties and known examples}\label{symm}

Let $i\in\{0,\dots,n-1\}$ and let $H\in {\mathcal{G}}(n,i)$ be fixed. Let $\mathcal{B}\subset {\mathcal{C}}^n$ be a class of nonempty compact sets in $\R^n$ and let $\mathcal{B}_H$ denote the subclass of members of $\mathcal{B}$ that are $H$-symmetric. We call a map ${\di}:{\mathcal{B}}\rightarrow{\mathcal{B}}_H$ an {\em $i$-symmetrization} on $\mathcal{B}$ (with respect to $H$).  If $K\in {\mathcal{B}}$, the corresponding set ${\di}K$ is called a {\em symmetral}.  We consider the following properties, where it is assumed that the class ${\mathcal{B}}$ is appropriate for the properties concerned and that they hold for all $K,L\in {\mathcal{B}}$.  Recall that $K^{\dagger}$ is the reflection of $K$ in $H$.

\smallskip

1. ({\em Monotonicity} or {\em strict monotonicity}) \quad $K\subset L \Rightarrow \di K\subset \di L$ (or $\di K\subset \di L$ and $K\neq L \Rightarrow \di K\neq \di L$, respectively).

2. ({\em $F$-preserving}) \quad $F(\di K)=F(K)$, where $F:{\mathcal{B}}\to [0,\infty)$ is a set function.  In particular, we can take $F=V_j$, $j=1,\dots,n$, the $j$th intrinsic volume, though we generally prefer to write {\em mean width preserving}, {\em surface area preserving}, and {\em volume preserving}  when $j=1$, $n-1$, and $n$, respectively.

3. ({\em Idempotent}) \quad $\di^2 K=\di(\di K)=\di K$.

4. ({\em Invariance on $H$-symmetric sets})\quad $K^{\dagger}=K\Rightarrow\di K=K$.

5. ({\em Invariance on $H$-symmetric spherical cylinders})\quad If $K=D_r(x)+s(B^n\cap H^{\perp})$, where $s>0$ and $D_r(x)\subset H$ is the $i$-dimensional ball with center $x$ and radius $r>0$, then $\di K=K$.

6. ({\em Projection invariance}) \quad $(\di K)|H=K|H$.

7. ({\em Invariance under translations orthogonal to $H$ of $H$-symmetric sets}) \quad If $K$ is $H$-symmetric and $y\in H^{\perp}$, then $\di(K+y)=\di K$.

8. ({\em Projection covariance}) \quad  $(\di K)|T=\di(K|T)$ for all nontrivial subspaces $T$ contained in $H^{\perp}$.

\smallskip

Invariance under translations orthogonal to $H$ of $H$-symmetric sets is satisfied by most natural symmetrizations, so in discussing examples we shall only mention this property when it does not hold.

Clearly invariance on $H$-symmetric sets implies invariance on $H$-symmetric spherical cylinders and idempotence.  Other less obvious relations between the eight properties are proved in Theorems~\ref{Invariance} and \ref{Gab3}.   Projection invariance and projection covariance are really only relevant when ${\mathcal{B}}\subset {\mathcal{K}}^n$. Other useful properties are considered in \cite{GHW} but will not be needed here.

Two special cases are of particular importance: $i=0$ and $i=n-1$.

If $i=0$, then $H=\{o\}$ and $0$-symmetrization is the same as the $o$-symmetrization discussed in \cite{GHW}.  Note that in this case, the definition of projection covariance above is consistent with that used in \cite{GHW}, since $H^{\perp}=\R^n$ and then (8) is equivalent to $(\di K)|T=\di(K|T)$ for all $T\in {\mathcal{G}}(n,j)$, $j\in \{1,\dots,n-1\}$.  Also, when $i=0$, the projection invariance property is trivially satisfied.

There are several useful examples of $0$-symmetrization, such as
{\em $p$th central symmetrization}, given for $K\in {\mathcal{K}}^n$ and $p\ge 1$ by
$$\di K=\Delta_p K=\left(2^{-1/p}K\right)+_p\left(2^{-1/p}(-K)\right).$$
Here $+_p$ denotes the general $L_p$ addition introduced in \cite{LYZ} (see also \cite[Example~6.7]{GHW} for an alternative approach), and the nontrivial fact that $\Delta_p K\in {\mathcal{K}}^n_s$ (in particular the convexity of this set) for $K\in {\mathcal{K}}^n$ follows from \cite[Theorem~5.1]{GHW}.  The $p$th central symmetrization is strictly monotonic, invariant on $H$-symmetric sets, and  projection covariant, as is easily verified using the fact that the operation $+_p$ is projection covariant as defined in Section~\ref{MS} below.   By Firey's inequality \cite[(78), p.~394]{Gar02}, the $j$th intrinsic volume is generally increased (meaning not decreased and not always preserved) by $\Delta_p$ for $j\in\{1,\dots,n\}$, except when $p=1$, in which case it is mean width preserving.  Except when $p=1$, $\Delta_p$ is not invariant under translations orthogonal to $H$ of $H$-symmetric sets.  When $p=1$, the subscript is dropped and
\begin{equation}\label{csdeff}
\di K=\Delta K=\frac{1}{2}K+\frac{1}{2}(-K)
\end{equation}
defines {\em central symmetrization} (see \cite[p.~106]{Gar06}).  The central symmetral $\Delta K$ differs from the ubiquitous {\em difference body} $DK=K+(-K)$ only by a dilatation factor of $1/2$.  There are many other important $o$-symmetrizations in convex geometry, for example, the projection body, intersection body, and centroid body operators, usually denoted by $\Pi K$, $IK$, and $\Gamma K$, respectively.  See, for example, \cite{Gar06} and \cite[Chapter~10]{Sch93}.

The other case of particular importance is $i=n-1$.  The prime example of an $(n-1)$-symmetrization is {\em Steiner symmetrization}. If $K\in {\mathcal{K}}^n$, the {\em Steiner symmetral} of $K$ with respect to $H\in {\mathcal{G}}(n,n-1)$ is the set $S_HK$ such that for each line $G$ orthogonal to $H$ and meeting $K$, the set $G\cap S_HK$ is a (possibly degenerate) closed line segment with midpoint in $H$ and length equal to that of $G\cap K$.  In this definition we have followed \cite[p.~536]{Sch93}, where the same definition is used for compact sets $K$, with length replaced by ${\mathcal{H}}^1$-measure; see also \cite[p.~62]{Gar06} and \cite[p.~169]{Gru07}.   On ${\mathcal{K}}^n$, Steiner symmetrization is strictly monotonic, volume preserving, invariant on $H$-symmetric sets, and projection invariant.  However, for $j\in \{1,\dots,n-1\}$, it generally reduces the $j$th intrinsic volume $V_j$ (see \cite[Satz~XI, p.~260]{Had57} or \cite[p.~587]{Sch93}), and it is not projection covariant, since the containment $(S_HK)|H^{\perp}\subset S_H(K|H^{\perp})$ is in general proper.

As an example of $i$-symmetrization, $i\in \{1,\dots,n-2\}$, we recall that when $K\in {\mathcal{K}}^n$, the {\em Schwarz symmetral} of $K$ with respect to $H\in {\mathcal{G}}(n,i)$ is the set $S_HK$ such that for each $(n-i)$-dimensional plane $G$ orthogonal to $H$ and meeting $K$, the set $G\cap S_HK$ is a (possibly degenerate) $(n-i)$-dimensional closed ball with center in $H$ and $(n-i)$-dimensional volume equal to that of $G\cap K$.  See \cite[p.~62]{Gar06} and also \cite[p.~178]{Gru07} (where the process is referred to as Schwarz rounding).  It is convenient to use the same notation for Steiner and Schwarz symmetrizations. On ${\mathcal{K}}^n$, Schwarz symmetrization is monotonic, volume preserving, idempotent, invariant on $H$-symmetric spherical cylinders, and projection invariant, but not strictly monotonic, invariant on $H$-symmetric sets, or projection covariant. (On  ${\mathcal{K}}^n_n$, Schwarz symmetrization is strictly monotonic.) Since it can be viewed as a limit of a sequence of Steiner symmetrizations, Schwarz symmetrization also generally reduces the $j$th intrinsic volume $V_j$ for $j\in \{1,\dots,n-1\}$.

We shall consider {\em Minkowski symmetrization} in the following general form.  Let $i\in \{0,\dots,n-1\}$ and let $H\in {\mathcal{G}}(n,i)$. The {\em Minkowski symmetral} of $K\in {\mathcal{K}}^n$ is defined by
\begin{equation}\label{Minks}
M_HK=\frac12 K+\frac12 K^{\dagger},
\end{equation}
where $K^{\dagger}$ is the reflection of $K$ in $H$. (Note that the case $i=0$ corresponds to $K^{\dagger}=-K$ and $M_HK=\Delta K$, the central symmetral.)  Minkowski symmetrization is strictly monotonic, and, since Minkowski addition commutes with projections, $K^{\dagger}|H=K|H$, and $K^{\dagger}|T=(K|T)^{\dagger}$ for all subspaces $T\subset H^{\perp}$, it is projection invariant and projection covariant.  It is clearly also invariant on $H$-symmetric sets.  Since the first intrinsic volume $V_1$ is linear with respect to Minkowski addition, $M_H$ is mean width preserving, but for $j\in\{2,\dots,n\}$, it generally increases the $j$th intrinsic volume $V_j$, by the Brunn-Minkowski inequality for quermassintegrals \cite[(74), p.~393]{Gar02} (see also \cite[Satz~XI, p.~260]{Had57}).

There is an extension of Minkowski symmetrization analogous to Schwarz symmetrization that we shall call {\em Minkowski-Blaschke symmetrization}, though it has been referred to by other names. For example, Bonnesen and Fenchel \cite[pp.~79--80]{BF87} call it stiffening and attribute it to Blaschke \cite[p.~137]{Bla49}.  If $i\in \{1,\dots,n-2\}$ and $H\in {\mathcal{G}}(n,i)$, the support function $h_K(u)$ of $K\in {\mathcal{K}}^n$ at a point $u\in S^{n-1}$ is replaced by the average of $h_K$ over the subsphere of $S^{n-1}$ orthogonal to $H$ and containing $u$.  More precisely, if $\overline{M}_HK$ denotes the Minkowski-Blaschke symmetral of $K$ and $u\in (H^{\perp}+x)\cap S^{n-1}$ for some $x\in H$, then
$$
h_{\overline{M}_HK}(u)=\frac{1}{(n-i)\kappa_{n-i}}
\int_{(H^{\perp}+x)\cap S^{n-1}}h_K(v)\,dv.
$$
One can check that $\overline{M}_HK\in {\mathcal{K}}^n$ and $\overline{M}_HK$ is rotationally symmetric with respect to $H$.  Minkowski-Blaschke symmetrization is strictly monotonic, mean width preserving (as can be shown by integration in spherical coordinates), idempotent, invariant on $H$-symmetric spherical cylinders, and projection invariant, but it is not invariant on $H$-symmetric sets or projection covariant.

If $H\in {\mathcal{G}}(n,i)$, $i\in \{0,\dots,n-1\}$, we define the {\em fiber symmetral} $F_HK$ of $K\in {\mathcal{K}}^n$ with respect to $H$ by
\begin{equation}\label{Fiber}
F_HK=\left(\frac12\circ K\right)\nplus_H \left(\frac12 \circ K^{\dagger}\right),
\end{equation}
where the fiber combination is defined by (\ref{Fiberdef}).  Then $F_HK=M_HK=\triangle K$ when $i=0$ and $F_HK=S_HK$ when $i=n-1$.  The latter was observed by McMullen \cite{McM} and for this reason we regard the fiber symmetral as known, though the general definition (\ref{Fiber}) does not appear in \cite{McM}.

Observing that $F_HK$ is the compact convex set whose sections orthogonal to $H$ are the Minkowski symmetrals of the corresponding sections of $K$, we take the opportunity to generalize McMullen's construction, as follows. Let $H\in {\mathcal{G}}(n,i)$, $i\in \{0,\dots,n-1\}$, and let $G\in {\mathcal{G}}(n,j)$, $j\in \{0,\dots,i\}$, be contained in $H$.  Let $K\in {\mathcal{K}}^n$ and dollowedefine $F_{H,G}K$ by
\begin{eqnarray}\label{fhjk}
F_{H,G}K&=&\bigcup_{x\in G}\left(\frac12 (K\cap (G^{\perp}+x))+\frac12(K\cap (G^{\perp}+x))^{\dagger}\right)\nonumber\\
&=&\bigcup_{x\in G}\left(\frac12 (K\cap (G^{\perp}+x))+\frac12(K^{\dagger}\cap (G^{\perp}+x))\right),
\end{eqnarray}
where $\dagger$ denotes reflection in $H$, as usual. Using (\ref{fhjk}), it is straightforward to show that $F_{H,G}K\in {\mathcal{K}}^n$ and that $F_{H,G}K$ is $H$-symmetric.  Note that $F_HK=F_{H,H}K$.  We shall use the same term, fiber symmetrization, for the map that takes $K$ to $F_{H,G}K$.  Fiber symmetrization is strictly monotonic, invariant on $H$-symmetric sets, and projection invariant, but when $i\in \{1,\dots,n-1\}$ it is not  projection covariant, since the containment $(F_{H,G}K)|H^{\perp}\subset F_{H,G}(K|H^{\perp})$ is in general proper.

Finally, for $H\in {\mathcal{G}}(n,i)$, $i\in \{0,\dots,n-1\}$, we can define the {\em Blaschke symmetral} of $K\in {\mathcal{K}}^n_n$ by
$$
B_HK=\left(2^{-1/(n-1)} K\right)\,\sharp\,\left(2^{-1/(n-1)} K^{\dagger}\right).
$$
Here $\sharp$ denotes Blaschke addition and $K\,\sharp\,L$ is a convex body such that the surface area measures satisfy $S(K\,\sharp\,L,\cdot)=S(K,\cdot)+S(L,\cdot)$. Thus we may equivalently define $B_HK$ by
\begin{equation}\label{Blsym}
S(B_HK,\cdot)=\frac12 S(K,\cdot)+\frac12 S(K^{\dagger},\cdot).
\end{equation}
When $i=0$, we have $K^{\dagger}=-K$ and then the body $B_HK$ is often called the {\em Blaschke body} of $K$ and denoted by $\nabla K$; see, for example, \cite[p.~116]{Gar06}.  Of course, (\ref{Blsym}) only defines $B_HK$ up to translation; see \cite{GPS} for a discussion about the positions for the Blaschke sum chosen in the literature.  We define the Blaschke sum so that the centroids of $B_HK$ and $K|H$ coincide, in which case $B_H$ is invariant on $H$-symmetric sets.  When $n=2$, then up to translation and on ${\mathcal{K}}^n_n$, $B_H$ coincides with $\Delta$ ($i=0$) or $M_H$ ($i=1$), whose properties have already been discussed.  When $n\geq3$, this is not the case and $B_H$ is neither monotonic nor (except when $i=0$) projection invariant, regardless of the position chosen for the Blaschke sum, as we show in the next theorem.  Also, projection covariance is not defined since the domain of $B_H$ is ${\mathcal{K}}^n_n$.   Blaschke symmetrization preserves surface area. It is a consequence of the Kneser-S\"{u}ss inequality \cite[(B.32), p.~423]{Gar06}, \cite[Theorem~8.2.3]{Sch93} that Blaschke symmetrization generally increases volume.

\begin{thm}\label{Blassym}
Let $H\in {\mathcal{G}}(n,i)$, $i\in\{0,\dots,n-1\}$.  If $i=0$, Blaschke symmetrization $B_H$ in $\R^n$, $n\ge 3$, is not monotonic and if $i\in\{1,\dots,n-1\}$, it is not projection invariant (and therefore, by Theorem~\ref{Invariance} below, also not monotonic).
\end{thm}

\begin{proof}
Let $T^n$ be an $n$-dimensional cone in $\R^n$ with centroid at the origin, $x_n$-axis as its axis, and radius and height (i.e., width in the direction $e_n$) both equal to 1.  Suppose initially that $i=0$.  We claim that when $n\ge 3$, the height of $B_HT^n=\nabla T^n$, the Blaschke body of $T^n$, is less than 1.  Suppose the claim is true. Let $0<s<1$ and let $L_s\subset T^n$ be the spherical cylinder with base of radius $s$ contained in the base of $T^n$, the $x_n$-axis as its axis, and with maximal height $w=w(s)$. The set $L_s$ is centrally symmetric, so $B_HL_s=\nabla L_s$ is a translate of $L_s$ and the height of $\nabla L_s$ is $w$; since $w\to 1$ as $s\to 0$, when $s$ is sufficiently small it is not possible that $\nabla L_s\subset \nabla T^n$.  This proves the result when $i=0$.

To prove the claim, let $n\ge 3$ and recall that the surface area of the curved part of the boundary of an $n$-dimensional cone of radius $r$ and height $h$ is $r^{n-2}\sqrt{h^2+r^2}\kappa_{n-1}$.  Therefore the surface area of the curved part of the boundary of $T^n$ is $\sqrt{2}\kappa_{n-1}$, while the area of the base of $T^n$ is $\kappa_{n-1}$.  The surface area measure $S(T^n,\cdot)$ consists of a point mass at $-e_n$ and a multiple of $(n-2)$-dimensional Lebesgue measure on the $(n-2)$-dimensional sphere of latitude in $S^{n-1}$ whose points have vertical angle $\pi/4$ with the positive $x_n$-axis.  From this and (\ref{Blsym}) it is easy to see that $\nabla T^n$ is an $o$-symmetric truncated double cone of radius $a$, say, with the $x_n$-axis as axis, such that the top of $\nabla T^n$ is an $(n-1)$-dimensional ball $B$ of radius $h$ contained in the plane $\{x_n=a-h\}$, for some $0<h<a$.  By (\ref{Blsym}), $V_{n-1}(B)=\kappa_{n-1}/2$, whence $h=2^{-1/(n-1)}$, and the surface area of the curved part of the boundary of $\nabla T^n$ contained in $\{x_n\ge 0\}$ is $\sqrt{2}\kappa_{n-1}/2$.  From the latter we see that
$$\sqrt{2}a^{n-1}\kappa_{n-1}-\sqrt{2}h^{n-1}\kappa_{n-1}=
\sqrt{2}\kappa_{n-1}/2$$
and hence $a=1$.  Thus the height of $\nabla T^n$ is
$2(a-h)=2(1-2^{-1/(n-1)})$, which is less than 1 when $n\ge 3$.  This proves the claim.

Let $i\in \{1,\dots,n-1\}$ and let $H$ be the subspace of $\R^n$ spanned by $e_1,\dots,e_i$.  Identifying $e_1^{\perp}$ with $\R^{n-1}$ in the natural way, let $T^{n-1}\subset e_1^{\perp}$ be as above, i.e., an $(n-1)$-dimensional cone with the $x_n$-axis as its axis, base an $(n-2)$-dimensional ball of radius 1 with center on the $x_n$-axis, height 1, and centroid at the origin.  Let $K=[-1/2,1/2]\times T^{n-1}$.  The cylinder $K$ has centroid at the origin and is such that $K^{\dagger}$, the reflection of $K$ in $H$, coincides with $-K$, so that $B_HK=\nabla K$.  (When $n=3$, $K$ is a triangular prism with its two triangular facets having edge lengths $\sqrt{2}$, $\sqrt{2}$, and 2.) The set $K$ has four facets: two $(n-1)$-dimensional spherical cylinders with outer unit normals $\pm e_n$ and volume $\kappa_{n-2}$, and two $(n-1)$-dimensional cones with outer unit normals $\pm e_1$ and volume $\kappa_{n-2}/(n-1)$.  By formulas already employed in the case $i=0$, the remaining curved (except when $n=3$) part of the boundary of $K$ in $\{x_n\ge 0\}$ has $(n-1)$-dimensional volume $\sqrt{2}\kappa_{n-2}$.

From considerations similar to those for the case $i=0$, we see that $\nabla K=[-b/2,b/2]\times D$, for some $b>0$, where $D\subset e_1^{\perp}$ is an $o$-symmetric $(n-1)$-dimensional double truncated cone of radius $a$, say, with the $x_n$-axis as its axis, such that the top of $D$ is an $(n-2)$-dimensional ball $B$ of radius $h$ contained in the plane $\{x_n=a-h\}$, for some $0<h<a$. The set $\nabla K$ also has four facets: two $(n-1)$-dimensional spherical cylinders with outer unit normals $\pm e_n$ and volume $h^{n-2}\kappa_{n-2}b$, and two $(n-1)$-dimensional double truncated cones with outer unit normals $\pm e_1$ and volume $2(a^{n-1}-h^{n-1})\kappa_{n-2}/(n-1)$.  (When $n=3$, $\nabla K$ is a hexagonal cylinder with eight facets.)  The remaining curved (except when $n=3$) part of the boundary of $\nabla K$ in $\{x_n\ge 0\}$ has $(n-1)$-dimensional volume $\sqrt{2}(a^{n-2}-h^{n-2})\kappa_{n-2}b$. From these facts and (\ref{Blsym}), we obtain $h^{n-2}b=1/2$, $2(a^{n-1}-h^{n-1})=1$, and $2(a^{n-2}-h^{n-2})b=1$.  Solving these three equations, we find that
$$b=2^{-\frac{1}{n-1}}(2^{\frac{n-1}{n-2}}-1)^{\frac{n-2}{n-1}}.$$
It is not hard to see that $b>1$, which implies that the width of $\nabla K$ in the direction $e_1$ is greater than that of $K$.  It follows that $(\nabla K)|H \neq K|H$.
\end{proof}

It is not possible to generalize the definitions of $M_H$ and $B_H$ in a straightforward way to obtain $V_j$-preserving symmetrizations for $j\in \{2,\dots,n-2\}$.  Indeed, by \cite[Theorem~3.1]{Sch94}, if $j\in \{2,\dots,n-2\}$ and $K$ is an $n$-dimensional convex polytope in $\R^n$ such that $K$ and $K^{\dagger}$ do not have non-trivial faces contained in parallel hyperplanes, then the sum $S_j(K,\cdot)+S_j(K^{\dagger},\cdot)$ of the $j$th area measures of $K$ and $K^{\dagger}$ is not the $j$th area measure of a compact convex set.  In this connection, see also Problem~\ref{prob0}.

Table~\ref{newtable} summarizes the properties of the symmetrizations discussed in this section.

\begin{table}
\begin{center}
\begin{tabular}{|c|c l|c|c|c|c|c|c|c|c|} \hline
 & & &\rotatebox[origin=c]{90}{Monotonic} & \rotatebox[origin=c]{90}{$V_j$-preserving} & \rotatebox[origin=c]{90}{Idempotent} & \rotatebox[origin=c]{90} {Inv. $H$-sym. sets} & \rotatebox[origin=c]{90}{Inv.~$H$-sym.~sph.~cyl.} & \rotatebox[origin=c]{90}{Projection inv.} & \rotatebox[origin=c]{90}{\ Inv. translations\ } & \rotatebox[origin=c]{90}{Projection cov.} \\
Name & Symbol & &1&2&3&4&5&6&7&8 \\ \hline
Central & $\Delta$ & & s\checkmark & $V_1$ & \checkmark & \checkmark & \checkmark & \checkmark & \checkmark & \checkmark \\ \hline
$p$th Central &$\Delta_p$, & $p>1$ & s\checkmark & \small{\ding{53}} & \checkmark & \checkmark & \checkmark & \checkmark & \small{\ding{53}} & \checkmark \\ \hline
Steiner &$S_H$, & $i=n-1$ & s\checkmark & $V_n$ & \checkmark & \checkmark & \checkmark & \checkmark & \checkmark & \small{\ding{53}} \\ \hline
Schwarz & $S_H$, & $1\le i\le n-2$ & \checkmark & $V_n$ & \checkmark & \small{\ding{53}} & \checkmark & \checkmark & \checkmark & \small{\ding{53}} \\ \hline
Minkowski & $M_H$ &  & s\checkmark & $V_1$ & \checkmark & \checkmark & \checkmark & \checkmark & \checkmark & \checkmark \\ \hline
Minkowski-Blaschke & $\overline{M}_H$, & $1\leq i\leq n-2$ & s\checkmark & $V_1$ & \checkmark & \small{\ding{53}} & \checkmark & \checkmark & \checkmark & \small{\ding{53}} \\ \hline
Fiber & $F_{H}$, & $1\leq i\leq n-2$ & s\checkmark & \small{\ding{53}} & \checkmark & \checkmark & \checkmark & \checkmark & \checkmark & \small{\ding{53}} \\ \hline
Blaschke & $B_H$, & $n\geq 3$ & \small{\ding{53}} & $V_{n-1}$ & \checkmark & \checkmark & \checkmark & \small{\ding{53}} & \checkmark & \small{\ding{53}} \\ \hline
\end{tabular}
\end{center}
\vspace{.2in}
\caption{Properties, numbered as in Section~\ref{symm}, of previously known symmetrizations,  where s\checkmark indicates strictly monotonic.}
\label{newtable}
\end{table}

\section{Projection covariant symmetrizations and $M$-symmetrization}\label{MS}

In this section we discuss projection covariant symmetrizations and introduce a further process that we call $M$-symmetrization.  We shall need a little terminology from \cite{GHW}. A binary operation $*:\left({\mathcal{K}}^n\right)^2\rightarrow {\mathcal{K}}^n$ is called {\em  $GL(n)$ covariant} if $\phi(K*L)=\phi K*\phi L$ for each $\phi\in GL(n)$ and all $K,L\in {\mathcal{K}}^n$, and {\em projection covariant} if $(K*L)|T=(K|T)*(L|T)$ for all $K,L\in {{\mathcal{K}}^n}$ and $T\in {\mathcal{G}}(n,j)$, $1\le j\le n-1$.  It is {\em monotonic} if $K\subset K'$ and $L\subset L'$ imply $K*L\subset K'*L'$ and {\em continuous} if $K_m\rightarrow M$ and $L_m\rightarrow N$ imply $K_m*L_m\rightarrow M*N$ in the Hausdorff metric as $m\rightarrow\infty$, where all sets are in ${\mathcal{K}}^n$.

\begin{thm}\label{projM}
Let $i\in \{0,\dots,n-1\}$, let $H\in {\mathcal{G}}(n,i)$, and let $*:\left({\mathcal{K}}^n\right)^2\rightarrow {\mathcal{K}}^n$ be commutative and $GL(n)$ covariant.  For $c>0$,
\begin{equation}\label{dagsym}
\di K=(cK)*(cK^{\dagger})
\end{equation}
for all $K\in {\mathcal{K}}^n$, defines an $i$-symmetrization $\di:{\mathcal{K}}^n\rightarrow {\mathcal{K}}^n_{H}$ which is also monotonic if $*$ is monotonic.  If in addition $*$ is continuous, then $\di$ is projection covariant.  If $*$ satisfies $(cK)*(cK)=K$ for all $K\in{\mathcal{K}}^n$, then $\di$ is invariant on $H$-symmetric sets.
\end{thm}

\begin{proof}
Since reflection in $H$ is a transformation in $GL(n)$, the commutativity and $GL(n)$ covariance of $*$ yields
$$(\di K)^{\dagger}=((cK)*(cK^{\dagger}))^{\dagger}=(cK)^{\dagger}*\left((cK^{\dagger})
^{\dagger}
\right)=(cK)^{\dagger}*(cK)=(cK)*(cK^{\dagger})=\di K,$$
so $\di:{\mathcal{K}}^n\rightarrow {\mathcal{K}}^n_H$ is an $i$-symmetrization. It is clear that if $*$ is monotonic, then $\di$ is monotonic.

Suppose that $*$ is also continuous.  Then, by \cite[Lemma~4.1]{GHW}, $*$ is projection covariant.  If $T$ is a nontrivial subspace contained in $H^{\perp}$, this and the commutativity of $*$ imply that
$$(\di K)|T=((cK)*(cK^{\dagger}))|T=((cK)|T)*((cK^{\dagger})|T)=(c(K|T))*(c(K|T)
^{\dagger})
=\di(K|T),$$
so $\di$ is projection covariant.

The last assertion in the statement of the theorem follows immediately from (\ref{dagsym}).
\end{proof}

Let $c>0$, let $H\in {\mathcal{G}}(n,i)$, $i\in \{0,\dots,n-1\}$, let $M\subset \R^2$, and for $K\in {\mathcal{K}}^n$, define
\begin{equation}\label{Msymme}
\di_{M,c}K=(cK)\oplus_M (cK^{\dagger}),
\end{equation}
where $\oplus_M$ is $M$-addition, defined by (\ref{Mdef}).  The following lemma will be useful in establishing the properties of $\di_{M,c}$.

\begin{lem}\label{lem_lpluslequall}
Let $c>0$, let $L\in\mathcal{K}^n$, and let $M\in {\mathcal{K}}^2$ be contained in $[0,\infty)^2$ and symmetric with respect to $\{(x_1,x_2)\in \R^2:x_1=x_2\}$.  If $M\subset[(1/c,0), (0,1/c)]$, then
\begin{equation}\label{lpluslequall}
 (cL)\oplus_M (cL)=L,
\end{equation}
and the converse is true when $o\notin L$.

If $o\in L$ and $o\in M$, then \eqref{lpluslequall} holds if and only if  $M\subset \conv\{o, (1/c,0), (0,1/c)\}$ and  $M\cap[(1/c,0), (0,1/c)]\neq\emptyset$.
\end{lem}

\begin{proof}
We have
\begin{equation}\label{explanation_lplusl}
(cL)\oplus_M (cL)=\{acx+bcy:x,y\in L, (a,b)\in M\}=\bigcup\{c(a+b)L: (a,b)\in M\}.
\end{equation}
Thus $M\subset[(1/c,0), (0,1/c)]$ implies $(cL)\oplus_M (cL)=L$. To prove the converse when $o\notin L$, we argue by contradiction. Assume that $(a,b)\in M$ is such that $c(a+b)< 1$. If $x\in L$ has minimum distance from $o$, then $c(a+b)x\in\left((cL)\oplus_M (cL)\right)\setminus L$, so \eqref{lpluslequall} fails. A similar argument applied to an $x\in L$ of maximal distance from $o$ shows that if $(a,b)\in M$ is such that $c(a+b)>1$, then \eqref{lpluslequall} fails.  Therefore $c(a+b)=1$ for each $(a,b)\in M$, yielding $M\subset[(1/c,0), (0,1/c)]$.

Now suppose that $o\in L$ and $o\in M$.  If $M\subset \conv\{o, (1/c,0), (0,1/c)\}$ and $M\cap[(1/c,0), (0,1/c)]\neq\emptyset$, then we see from
\eqref{explanation_lplusl} that \eqref{lpluslequall} holds. Conversely, assume that \eqref{lpluslequall} holds and let $x\in L$ be of maximal distance from $o$. If $(a,b)\in M$ is such that $c(a+b)>1$, then $c(a+b)x\in ((cL)\oplus_M (cL))\setminus L$, a contradiction showing that
$M\subset \conv\{o, (1/c,0), (0,1/c)\}$.  If $M\cap[(1/c,0), (0,1/c)]=\emptyset$ then $x\in L\setminus \left((cL)\oplus_M (cL)\right)$, again contradicting \eqref{lpluslequall}.
\end{proof}

\begin{thm}\label{corrM}
Let $c>0$, let $H\in {\mathcal{G}}(n,i)$, $i\in \{0,\dots,n-1\}$, and let $M\in {\mathcal{K}}^2$ be contained in $[0,\infty)^2$ and symmetric with respect to $\{(x_1,x_2)\in \R^2:x_1=x_2\}$.

{\rm{(i)}} Equation \eqref{Msymme} defines a monotonic and projection covariant $i$-symmetrization $\di_{M,c}:{\mathcal{K}}^n\rightarrow {\mathcal{K}}^n_{H}$.

{\rm{(ii)}} If $M\subset(0,\infty)^2$, then $\di_{M,c}$  is strictly monotonic.

{\rm{(iii)}} The symmetrization $\di_{M,c}$  is idempotent if and only if either $M\subset[(1/c,0), (0,1/c)]$ or $o\in M$, $M\subset \conv\{o, (1/c,0), (0,1/c)\}$, and $M\cap[(1/c,0), (0,1/c)]\neq\emptyset$.

{\rm{(iv)}} The inclusion $M\subset[(1/c,0), (0,1/c)]$ is equivalent to any of the following conditions: $\di_{M,c}$ is invariant on $H$-symmetric sets, or $\di_{M,c}$ is invariant on $H$-symmetric spherical cylinders, or $\di_{M,c}$ is projection invariant.

{\rm{(v)}} The symmetrization $\di_{M,c}$ is invariant under translations orthogonal to $H$ of $H$-symmetric sets if and only if $M\subset\{(x_1,x_2)\in \R^2 : x_1=x_2\}$.
\end{thm}

\begin{proof}
(i) By \cite[Theorem~6.1(i)]{GHW}, $\oplus_M:\left({\mathcal{K}}^n\right)^2\rightarrow {\mathcal{K}}^n$ if and only if $M\in {\mathcal{K}}^2$ and $M$ is contained in one of the four quadrants of $\R^2$.  Also, as noted in \cite[Section~6]{GHW}, $\oplus_M$ is continuous and $GL(n)$-covariant and hence projection covariant. Moreover, it follows from the definition of $\oplus_M$ that it is monotonic, and $\oplus_M$ is commutative if and only if $M$ is symmetric with respect to $\{(x_1,x_2)\in \R^2:x_1=x_2\}$.  Therefore (i) follows directly from Theorem~\ref{projM}.

(ii) Let $K,L\in\mathcal{K}^n$ and let $x\in \R^n$. By \cite[Theorem 6.5]{GHW},
\begin{equation}\label{GHWform}
h_{K\oplus_M L}(x)=h_M(h_K(x), h_L(x)),
\end{equation}
implying that
\begin{equation}\label{support_msymm}
h_{\di_{M,c}K}(x)=h_M(ch_K(x),ch_{K^\dag}(x)).
\end{equation}
Suppose that $K\subset L$ and $K\neq L$. There exists $z\in\R^n$ such that
\begin{equation}\label{ineq_strict_inclusion}
 h_K(z)<h_L(z)\quad\text{and}\quad h_{K^\dag}(z)\leq h_{L^\dag}(z).
\end{equation}
Let $(a,b)\in M$ be such that
$$
h_M(ch_K(z),ch_{K^\dag}(z))=(a,b)\cdot(ch_K(z),ch_{K^\dag}(z)).
$$
Since $a,b>0$, \eqref{support_msymm} and \eqref{ineq_strict_inclusion} yield
$$
h_{\di_{M,c}K}(x)=(a,b)\cdot(ch_K(x),ch_{K^\dag}(x))<(a,b)\cdot(ch_L(x),c h_{L^\dag}(x))\leq h_{\di_{M,c}L}(x)
$$
and hence $\di_{M,c}K\neq \di_{M,c}L$.

(iii) Let $K\in\mathcal{K}^n$. Since $\di_{M,c} K$ is $H$-symmetric, we have
\begin{equation}\label{explanation_idempotence}
 \di_{M,c}^2 K=(c\di_{M,c} K)\oplus_M \left(c(\di_{M,c} K)^\dag\right)=(c\di_{M,c} K)\oplus_M (c\di_{M,c} K).
\end{equation}
If $M\subset[(1/c,0),(0,1/c)]$, the idempotence of $\di_{M,c}$ follows from this and Lemma~\ref{lem_lpluslequall} with $L=\di_{M,c}K$. If $o\in M$, $M\subset \conv\{o, (1/c,0), (0,1/c)\}$, and $M\cap[(1/c,0), (0,1/c)]\neq\emptyset$, then $o\in\di_{M,c} K$ and again the  idempotence of $\di_{M,c}$ follows from \eqref{explanation_idempotence} and Lemma~\ref{lem_lpluslequall} with $L=\di_{M,c}K$.
To prove the converse, suppose that $\di_{M,c}$ is idempotent. If  $o\notin\di_{M,c} K$ for some $K\in \mathcal{K}^n$, then $M\subset[(1/c,0), (0,1/c)]$, by Lemma~\ref{lem_lpluslequall} with $L=\di_{M,c}K$. Otherwise, we have $o\in\di_{M,c}K$ for all $K\in \mathcal{K}^n$.  We claim that $o\in M$.  Indeed, suppose on the contrary that $d=\min\{a+b : (a,b)\in M\}>0$. If $H=\lin\{e_1,\dots,e_i\}$ and $K\subset\{x\in\R^n : x\cdot e_1\geq 1\}$, then $K^\dag\subset\{x\in\R^n : x\cdot e_1\geq 1\}$. Then for each $x\in K$, $y\in K^\dag$, and $(a,b)\in M$, we have
$$
 c(ax+by)\cdot e_1\geq c(a+b)\geq cd.
$$
Therefore $\di_{M,c}K\subset \{x\in\R^n : x\cdot e_1\geq cd\}$, which since $cd>0$ contradicts $o\in\di_{M,c}K$ and proves the claim. Now since $o\in M$, Lemma~\ref{lem_lpluslequall} with $L=\di_{M,c}K$ implies that $M\subset \conv\{o, (1/c,0), (0,1/c)\}$ and $M\cap[(1/c,0), (0,1/c)]\neq\emptyset$.

(iv) Suppose that $M\subset[(1/c,0), (0,1/c)]$. If $K$ is an $H$-symmetric set, then
$$
 \di_{M,c}K=(cK)\oplus_M (cK^{\dagger})=(cK)\oplus_M (cK).
$$
Then Lemma~\ref{lem_lpluslequall} with $L=K$ yields $\di_{M,c}K=K$.  Therefore $\di_{M,c}$ is invariant on $H$-symmetric sets and hence also on $H$-symmetric spherical cylinders.  Now suppose that $\di_{M,c}$ is invariant on $H$-symmetric spherical cylinders. Let $K$ be an $H$-symmetric spherical cylinder such that $o\not\in K$.  Then (\ref{lpluslequall}) holds with $L=K$ and  Lemma~\ref{lem_lpluslequall} implies that $M\subset[(1/c,0), (0,1/c)]$.  It remains to prove (iv) for projection invariance, which follows easily from similar arguments and the formulas
$$
\left((cK)\oplus_M (cK^{\dagger})\right)|H=\left(cK|H\right)\oplus_M \left(cK^{\dagger}|H\right)=\left(cK|H\right)\oplus_M \left( cK|H\right).
$$

(v) If $M\subset\{(x_1,x_2)\in\R^2 : x_1=x_2\}$, then $h_M$ is constant on each line orthogonal to $\{(x_1,x_2)\in\R^2 : x_1=x_2\}$. Therefore, if $K\in\mathcal{K}^n$ is an $H$-symmetric set and $y\in H^\perp$, by (\ref{GHWform}) we have
\begin{equation}\label{support_constant_on_lines}
\begin{aligned}
h_{\di_{M,c}(K+y)}(x)&=h_M(ch_{K+y}(x),ch_{(K+y)^\dag}(x))\\
&=h_M(ch_K(x)+cy\cdot x,ch_K(x)-cy\cdot x)\\
&=h_M(ch_K(x),ch_K(x))=h_{\di_{M,c}K}(x).
\end{aligned}
\end{equation}
Thus $\di_{M,c}$ is invariant under translations orthogonal to $H$ of $H$-symmetric sets. Conversely, suppose that $\di_{M,c}$ is invariant under translations orthogonal to $H$ of $H$-symmetric sets.  Let $x\in\R^n\setminus \left(H\cup H^\perp\right)$ and let $K$ be a ball with center in $H$, supported by $x^{\perp}$ and contained in $\{y\in\R^n : y\cdot x\leq 0\}$.  Then $h_K(x)=0$. Choosing $y\in H^\perp$ such that $x\cdot y=\pm 1/c$ and substituting into \eqref{support_constant_on_lines}, we conclude that $h_M(1,-1)=h_M(-1,1)=0$ and hence that
$M\subset\{(x_1,x_2)\in\R^2 : x_1=x_2\}$.
\end{proof}

We call an $i$-symmetrization $\di_{M,c}$, defined by (\ref{Msymme}) with $c$ and $M$ satisfying the hypotheses of Theorem~\ref{corrM}, an {\em $M$-symmetrization}.  Examples include Minkowski symmetrization (when $\oplus_M=+$ and $c=1/2$) and $p$th central symmetrization (when $i=0$, $\oplus_M=+_p$, and $c=2^{-1/p}$).  Before discussing others, we present the following result concerning projection covariant $i$-symmetrizations, which shows that they are somewhat special and explains their limited role in this paper.  The proof is essentially the same as that of \cite[Theorem~8.2]{GHW}, which establishes the case $i=0$.

\begin{prop}\label{pcs}
Let $i\in \{0,\dots,n-2\}$, let $H\in {\mathcal{G}}(n,i)$, and let $\di:{\mathcal{K}}^n\to {\mathcal{K}}^n_H$ be an $i$-symmetrization.  If $\di$ is projection covariant, then there is a compact convex set $M$ in $\R^2$, symmetric with respect to $\{(x_1,x_2)\in \R^2:x_1=x_2\}$, such that
\begin{equation}\label{11}
h_{\di K}(x)=h_M\left(h_K(x),h_K(-x)\right)
\end{equation}
for all $K\in {\mathcal{K}}^n$ and $x\in H^{\perp}$.
\end{prop}

\begin{proof}
Suppose that $\di$ is projection covariant. Without loss of generality, we may assume that $H^{\perp}=\lin\{e_1,\dots,e_{n-i}\}$ and $H=\lin\{e_{n-i+1},\dots,e_n\}$. Since $n-i\ge 2$, the proof of \cite[Theorem~8.2]{GHW} can be followed with $n$ there replaced by $n-i$, identifying $H^{\perp}$ by $\R^{n-i}$ in that proof in the natural way, leading to (\ref{11}).
\end{proof}

When $M=\{(1/2,1/2)\}$, for example, (\ref{11}) is equivalent to
$$h_{\di K}(x)=\frac12 h_K(x)+\frac12 h_{K}(-x)$$
for all $K\in {\mathcal{K}}^n$ and $x\in H^{\perp}$.  Then $\di K$ is just the central symmetral $\Delta K$ when $i=0$.

The converse of Proposition~\ref{pcs} holds when $i=0$ or if it is assumed in addition to (\ref{11}) that $K\subset H^{\perp}$ implies that $\di K\subset H^{\perp}$, but in general it is false.  For example, let $i\in \{1,\dots,n-2\}$ and let $\di K=M_HK+(B^n\cap H)$ for all $K\in {\mathcal{K}}^n$, as in Example~\ref{ex2} below.  Then (\ref{11}) holds with $M=\{(1/2,1/2)\}$, since if $x\in H^{\perp}$, then $h_{K^{\dagger}}(x)=h_K(-x)$ and $h_{B^n\cap H}(x)=0$.  However, $\di$ is not projection covariant, since if $T$ is a nontrivial subspace contained in $H^{\perp}$, then $(\di B^n)|T=B^n|T$, but
$$\di(B^n|T)=(B^n|T)+(B^n\cap H).$$

We warn the reader that the right-hand side of (\ref{11}) does not necessarily define a support function, even when $i=0$ and $M$ is the unit ball in $l^2_p$ with $p>1$; see \cite[p.~2334]{GHW}.  Despite this, if in (\ref{11}) we take $i=0$ and $M$ to be the part of the unit ball in $l^2_{p'}$ in $[0,\infty)^2$, where $1/p+1/p'=1$ and $p>1$, then $\di K=\Delta_pK$; see \cite[Example~6.7]{GHW}.

We now consider other possibilities for $M$-symmetrizations.  If $M\subset [(1/c,0),(0,1/c)]$, then by Theorem~\ref{corrM}(iv), $\di_{M,c}$ is invariant on $H$-symmetric sets and projection invariant. When $c=1$ and $M=\{(1/2,1/2)\}$, we retrieve Minkowski symmetrization as in (\ref{Minks}).  On the other hand, if we take $c=1$ and $M=[(1,0),(0,1)]$, then
$$\di_{M,1}K=\bigcup\{(1-t)x+ty: x\in K, y\in K^{\dagger}, 0\le t\le 1\}=\conv(K\cup K^{\dagger}),$$
as in Example~\ref{ex2nww}(ii) below.  A further choice for $M$ provides Example~\ref{ex1Gab} below.  These examples will find use in Section~\ref{idemp}.

\section{Further examples of symmetrizations}\label{further}

In this section we collect some further examples of symmetrizations that will be needed later.

Firstly, we introduce two new generalizations of Steiner and Minkowski symmetrization.  Let $H\in {\mathcal{G}}(n,i)$, $i\in\{1,\dots,n-1\}$. If $K\in {{\mathcal{K}}^n}$, we define the {\em inner rotational symmetral} $I_HK$ to be the set such that for each $(n-i)$-dimensional plane $G$ orthogonal to $H$ and meeting $K$, the set $G\cap I_HK$ is a (possibly degenerate) $(n-i)$-dimensional ball with center in $H$ and radius equal to that of the (possibly degenerate) largest $(n-i)$-dimensional ball contained in $G\cap K$.  It is easy to check that $I_HK$ is a compact convex set such that $I_HK\subset S_HK$, the Schwarz symmetral of $K$.  If $i=n-1$, then of course $I_HK=S_HK$, the Steiner symmetral. The symmetrization $I_H$ is monotonic, idempotent, invariant on $H$-symmetric spherical cylinders, and projection invariant, but not strictly monotonic or invariant on $H$-symmetric sets.  The inclusion $I_HK\subset S_HK$, which in general is strict unless $i=n-1$, and the fact that $S_H$ is $V_n$-preserving and generally reduces $V_j$ for $j\in \{1,\dots,n-1\}$, imply that $I_H$ generally reduces $V_j$ for $j\in \{1,\dots,n-1\}$ and also for $j=n$ when $i\in \{1,\dots,n-2\}$.

When $K$ is a convex body, $I_HK$ can also be viewed as the closure of the union of all $H$-symmetric spherical cylinders that have a translate orthogonal to $H$ contained in $K$.  It is rotationally symmetric with respect to $H$.

For $K\in {{\mathcal{K}}^n}$, we define the {\em outer rotational symmetral} $O_HK$ to be the intersection of all rotationally symmetric convex bodies for which some translate orthogonal to $H$ contains $K$.  The inclusion $M_HK\subset O_HK$ holds, as can be seen by taking $\di=M_H$ in Theorem~\ref{IdemGab2} below.  If $i=n-1$, then $O_HK=M_HK$, the Minkowski symmetral.  This can be deduced by taking $\di=O_H$ and $i=n-1$ in Corollary~\ref{IdemGab}.  The symmetrization $O_H$ is strictly monotonic, idempotent, and invariant on $H$-symmetric spherical cylinders, but not invariant on $H$-symmetric sets unless $i=n-1$.  That it is also projection invariant can be seen by noting that for each $K\in {{\mathcal{K}}^n}$ and suitably large $r>0$, the set $L=(K|H)\times r(B^n\cap H^{\perp})$ is a rotationally invariant compact convex set with $K\subset L$ and $L|H=K|H$.  The inclusion $M_HK\subset O_HK$, which in general is strict unless $i=n-1$, and the fact that $M_H$ is $V_1$-preserving and generally increases $V_j$ for $j\in \{2,\dots,n\}$, imply that $O_H$ generally increases $V_j$ for $j\in \{2,\dots,n\}$ and also for $j=1$ when $i\in \{1,\dots,n-2\}$.

We present some examples that will be useful in showing that the various assumptions in our results cannot be omitted.  (See Section~\ref{SS} for a few further examples needed only there.) For Example~\ref{ex3nww}, we shall need the following lemma.

\begin{lem}\label{VJ}
Let $H\in {\mathcal{G}}(n,i)$, $i\in\{1,\dots,n-1\}$, and let $K\in{\mathcal{K}}^n$.   Define $L_K = \conv(\overline{K}\cup (K|H))$, where $\overline{K}$ is the (possibly empty) closure of the union of all $H$-symmetric spherical cylinders contained in $K$.  Then for $j\in\{1,\dots,n\}$,
\begin{equation}\label{vjineq}
V_j(K|H)\le V_j(L_K)\le V_j(K).
\end{equation}
\end{lem}

\begin{proof}
From the definitions of $L_K$ and $I_HK$, we have $K|H\subset L_K\subset I_HK$.  Then (\ref{vjineq}) follows, since $V_j$ is an increasing set function and $I_H$ does not increase $V_j$.
\end{proof}

In Examples~\ref{ex1nwwmar}--\ref{ex1Gab} below, we assume for convenience that $H\in {\mathcal{G}}(n,i)$, $i\in \{1,\dots,n-1\}$ (even though $i=0$ is sometimes possible), and ${\mathcal{B}}={\mathcal{K}}^n$ or ${\mathcal{B}}={\mathcal{K}}^n_n$, and discuss the properties of the symmetrizations defined.  All symmetrizations are invariant under translations orthogonal to $H$ of $H$-symmetric sets unless it is stated otherwise. We omit mention of projection covariance since it will not be needed for the rest of the paper.

\begin{ex}\label{ex1nwwmar}
{\em For all $K\in {\mathcal{B}}$, let $\di K$ be the smallest $H$-symmetric spherical cylinder such that some translate orthogonal to $H$ contains $K$, where we temporarily allow the cylinder to be degenerate.  Then ${\di}:{\mathcal{B}}\rightarrow{\mathcal{B}}_H$ is monotonic, idempotent, and invariant on $H$-symmetric spherical cylinders, but not strictly monotonic or invariant on $H$-symmetric sets. Also, $\di$ is projection invariant if and only if $i=1$.
\qed}
\end{ex}

\begin{ex}\label{exSHcont}
{\em For all $K\in {\mathcal{B}}$, let $x_K$ and $y_K$ be the centroids of $K|H$ and $K|H^{\perp}$, respectively.  Let $r_K,s_K\ge 0$ be the largest numbers such that $D_K=r_K(B^n\cap H)+x_K\subset K|H$ and $E_K=s_K(B^n\cap H^{\perp})+y_K\subset K|H^{\perp}$. Define $\di K=D_K+E_K$. Then ${\di}:{\mathcal{B}}\rightarrow{\mathcal{B}}_H$ is idempotent and invariant on $H$-symmetric spherical cylinders, but not monotonic or invariant on $H$-symmetric sets.  Also, $\di$ is projection invariant if and only if $i=1$.\qed}
\end{ex}

\begin{ex}\label{ex1nww}
{\em For all $K\in {\mathcal{B}}$, let $\di K=M_HK+V_n(K\triangle K^{\ddag})B^n$, where $\triangle$ denotes the symmetric difference and $K^{\ddag}$ is the reflection of $K$ in the translate of $H$ containing the centroid of $K$. Then ${\di}:{\mathcal{B}}\rightarrow{\mathcal{B}}_H$ is invariant on $H$-symmetric sets but not projection invariant and hence, by Theorem~\ref{Invariance} below, not monotonic either (as is also not hard to see directly).
\qed}
\end{ex}

\begin{ex}\label{ex2}
{\em For all $K\in {\mathcal{B}}$, let $\di K=M_HK+(B^n\cap H^{\perp})$. Then ${\di}:{\mathcal{B}}\rightarrow{\mathcal{B}}_H$ is strictly monotonic and projection invariant, but not idempotent or invariant on $H$-symmetric spherical cylinders.\qed}
\end{ex}

\begin{ex}\label{lastex}
{\em For all $K\in {\mathcal{B}}$, let $C_K$ be the smallest $H$-symmetric spherical cylinder such that some translate orthogonal to $H$ contains $K$ (i.e., $C_K$ is the symmetral from Example~\ref{ex1nwwmar}) and define $\di K=(1/2)O_HK + (1/2)C_K$, where $O_HK$ is the outer rotational symmetral of $K$. Then ${\di}:{\mathcal{B}}\rightarrow{\mathcal{B}}_H$ is strictly monotonic and invariant on $H$-symmetric spherical cylinders, but not idempotent or invariant on $H$-symmetric sets.  Also, $\di$ is projection invariant if and only if $i=1$.
\qed}
\end{ex}

\begin{ex}\label{ex2vi}
{\em For all $K\in {\mathcal{B}}$, let $\di K=(K|H)+(B^n\cap H^{\perp})$. Then ${\di}:{\mathcal{B}}\rightarrow{\mathcal{B}}_H$ is monotonic, idempotent, and projection invariant, but not strictly monotonic or invariant on $H$-symmetric spherical cylinders.\qed}
\end{ex}

\begin{ex}\label{Vexlast}
{\em For all $K\in {\mathcal{B}}$, let $d_K$ be the distance between $K$ and $H$.  For $a\ge 0$, let $\phi_a:\R^n\to\R^n$ be defined by $\phi_a(x+y)=x+ay$, where $x\in H$ and $y\in H^{\perp}$.  Define ${\di}:{\mathcal{B}}\rightarrow{\mathcal{B}}_H$ by $\di K=\phi_{e^{-d_K}}F_HK$.  It is easy to check that $\di$ is strictly monotonic, invariant on $H$-symmetric sets, and projection invariant, but not invariant under translations orthogonal to $H$ of $H$-symmetric sets.}
\end{ex}

\begin{ex}\label{Vexlas2}
{\em For all $K\in {\mathcal{B}}$, define ${\di}:{\mathcal{B}}\rightarrow{\mathcal{B}}_H$ by $\di K=\phi_{e^{-V_n(K\triangle K^{\ddag})}}F_HK$, where $K\triangle K^{\ddag}$ and $\phi_a$, $a\ge 0$, are as in Examples~\ref{ex1nww} and~\ref{Vexlast}, respectively.  Then $\di$ is invariant on $H$-symmetric sets and projection invariant, but not monotonic.}
\end{ex}

\begin{ex}\label{Vex2}
{\em Let $j\in \{1,\dots,n\}$.  For all $K\in {\mathcal{B}}$, let $\di K=t_KB^n$, where $t_K\ge 0$ is chosen so that $V_j(\di K)=V_j(K)$. Then ${\di}:{\mathcal{B}}\rightarrow{\mathcal{B}}_H$ is monotonic (strictly monotonic if and only if ${{\mathcal{B}}}={\mathcal{K}}^n_n$ or $j=1$), $V_j$-preserving, and idempotent, but not invariant on $H$-symmetric spherical cylinders or projection invariant.\qed}
\end{ex}

\begin{ex}\label{ex2nww}
{\em Either define

\rm{(i)}  $\di K=\conv((K\cap K^{\dagger})\cup (K|H))$ for all $K\in {\mathcal{B}}$, or

\rm{(ii)} $\di K=\conv(K\cup K^{\dagger})$ for all $K\in {\mathcal{B}}$.

Then ${\di}:{\mathcal{B}}\rightarrow{\mathcal{B}}_H$ is monotonic and invariant on $H$-symmetric sets (and therefore projection invariant, by Theorem~\ref{Invariance} below), but not strictly monotonic, $V_j$-preserving for any $j\in \{1,\dots,n\}$, or invariant under translations orthogonal to $H$ of $H$-symmetric sets.\qed}
\end{ex}

\begin{ex}\label{ex3nww}
{\em Let $j\in \{1,\dots,n\}$. For $K\in {{\mathcal{B}}}$, define $L_K = \conv(\overline{K}\cup (K|H))$, where $\overline{K}$ is the (possibly empty) closure of the union of all $H$-symmetric spherical cylinders contained in $K$.  By (\ref{vjineq}), we may choose $t_K\ge 0$ such that $V_j(L_K+t_K(B^n\cap H^{\perp}))=V_j(K)$ and define
$$\di K=L_K+t_K(B^n\cap H^{\perp})=\conv(\overline{K}\cup (K|H))+t_K(B^n\cap H^{\perp}).$$
If ${{\mathcal{B}}}={\mathcal{K}}^n_n$, then clearly $\dim \di K=\dim L_K=n$ when $\overline{K}\neq\emptyset$.  If $\overline{K}=\emptyset$, then $L_K=K|H$ and $V_j(K)>V_j(L_K)$; then $t_K>0$ and again we have $\dim \di K=n$.  Hence $\di:{\mathcal{B}}\rightarrow {\mathcal{B}}_{H}$.
Then $\di$ is idempotent, $V_j$-preserving, invariant on $H$-symmetric spherical cylinders, and projection invariant.  However, $\di$ is not monotonic, as we show below, and neither invariant on $H$-symmetric sets nor invariant under translations orthogonal to $H$ of $H$-symmetric sets.

The fact that $\di$ is not monotonic when $j=1$ or $j=n$ is a consequence of Theorem~\ref{MinChar} and Corollary~\ref{Steiner} below, respectively, but can also be seen directly for all $j\in \{1,\dots,n\}$.  Let $C\subset H$ be an $i$-dimensional compact convex set and let $K=C+(B^n\cap H^{\perp})+y$, where $y\in H^{\perp}$ is such that $K\cap H=\emptyset$ and hence $\overline{K}=\emptyset$.  Let $L\in {\mathcal{K}}^n_n$ satisfy $K\subset L\subset (L|H)+(B^n\cap H^{\perp})+y$, where both inclusions are strict.  Since $K|H=C$, we obtain
$$\di K=C+t_K(B^n\cap H^{\perp})=C+(B^n\cap H^{\perp})=K-y,$$
because only $t_K=1$ gives $V_j(\di K)=V_j(K)$.  Also, since $\overline{L}=\emptyset$, we have
$$\di L=(L|H)+t_L(B^n\cap H^{\perp}).$$
If $t_L\ge 1$, then $\di L$ strictly contains $L-y$, so $V_j(\di L)>V_j(L)$. Therefore $t_L<1$, which implies that $\di K\not\subset \di L$.  \qed}
\end{ex}

\begin{ex}\label{ex6}
{\em Let $j\in \{1,\dots,n-i\}$ and let $L\in {\mathcal{K}}^n_s$ satisfy $L\subset H^{\perp}$ and $V_{j}(L)=V_j(B^n\cap H^{\perp})$. For $K\in {{\mathcal{B}}}$, define
$$\di K=\bigcup\{r_xL+x: x\in K|H\},$$
where $r_x\ge 0$ is chosen so that
$$V_{j}(r_xL)=V_{j}\left(K\cap (H^{\perp}+x)\right)$$
for all $x\in K|H$.  Then the Brunn-Minkowski inequality for quermassintegrals \cite[(74), p.~393]{Gar02} implies that $\di:{\mathcal{B}}\rightarrow {\mathcal{B}}_{H}$ is an $i$-symmetrization.  If $i=n-1$, then $\di$ is Steiner symmetrization and if $i\in \{1,\dots,n-2\}$, $j=n-i$, and $L=B^n\cap H^{\perp}$, then $\di$ is Schwarz symmetrization. Moreover, if $i\in \{1,\dots,n-2\}$, then $\di$ is monotonic (strictly monotonic if and only if ${{\mathcal{B}}}={\mathcal{K}}^n_n$ or $j=1$), volume preserving if $j=n-i$, idempotent, and projection invariant, but not invariant on $H$-symmetric spherical cylinders unless $L=B^n\cap H^{\perp}$ and not invariant on $H$-symmetric sets.
\qed}
\end{ex}

\begin{ex}\label{exoct2}
{\em Let $j\in \{1,\dots,n-1\}$ and let $K\in {\mathcal{B}}$. If $K=L+y$, where $L$ is $H$-symmetric and $y\in H^{\perp}$, then define $\di K=L$.  Otherwise, define $\di K=t_KB^n$, where $t_K\ge 0$ is chosen so that $V_j(\di K)=V_j(K)$. Then ${\di}:{\mathcal{B}}\rightarrow{\mathcal{B}}_H$ is $V_j$-preserving and invariant on $H$-symmetric sets, but not monotonic or projection invariant.\qed}
\end{ex}

\begin{ex}\label{exoct}
{\em Let $j\in \{1,\dots,n-1\}$.  For all $K\in {\mathcal{B}}$, let $x_K$ be the centroid of $K|H$.   Choose $r_K\ge 0$ such that $C_K=\conv((K|H)\cup( r_K(B^n\cap H^{\perp})+x_K))$ satisfies $V_j(C_K)=V_j(K)$, and define $\di K=C_K$. Then ${\di}:{\mathcal{B}}\rightarrow{\mathcal{B}}_H$ is $V_j$-preserving, idempotent, and projection invariant, but not monotonic or invariant on $H$-symmetric spherical cylinders.\qed}
\end{ex}

\begin{ex}\label{ex1Gab}
{\em  Let $H\in {\mathcal{G}}(n,i)$, $i\in\{1,\dots,n-1\}$, and let ${\mathcal{B}}={\mathcal{K}}^n$ or ${\mathcal{B}}={\mathcal{K}}^n_n$.  Define
${\di}:{\mathcal{B}}\rightarrow{\mathcal{B}}_H$ by (\ref{Msymme}) with $c=1$ and $M=[(1/4,3/4),(3/4,1/4)]$.  Theorem~\ref{corrM} implies that this definition is valid and that $\di$ is strictly monotonic, invariant on $H$-symmetric sets, and projection invariant, but not invariant under translations orthogonal to $H$ of $H$-symmetric sets.
\qed}
\end{ex}

Table~\ref{table2} summarizes the properties (other than projection covariance, which is irrelevant for the sequel) of symmetrizations in this section.  For simplicity a few special cases are ignored; for example, the symmetrizations in Examples~\ref{Vex2} and~\ref{ex6} are strictly monotonic when $j=1$.

\begin{table}
\begin{center}
\begin{tabular}{|c l|c|c|c|c|c|c|c|c|} \hline
 & &\rotatebox[origin=c]{90}{Monotonic} & \rotatebox[origin=c]{90}{$V_j$-preserving} & \rotatebox[origin=c]{90}{Idempotent} & \rotatebox[origin=c]{90} {Inv. $H$-sym. sets} & \rotatebox[origin=c]{90}{Inv.~$H$-sym.~sph.~cyl.} & \rotatebox[origin=c]{90}{Projection inv.} & \rotatebox[origin=c]{90}{\ Inv. translations\ }  \\
  & &1&2&3&4&5&6&7 \\ \hline
$I_H$, & $1\leq i\leq n-2$ & \checkmark & \small{\ding{53}} & \checkmark & \small{\ding{53}} & \checkmark & \checkmark & \checkmark  \\ \hline
$O_H$, & $1\leq i\leq n-2$ & s\checkmark& \small{\ding{53}} & \checkmark & \small{\ding{53}} & \checkmark & \checkmark & \checkmark  \\ \hline
\ref{ex1nwwmar}, & $i\neq 1$ & \checkmark & \small{\ding{53}} & \checkmark & \small{\ding{53}} & \checkmark & \small{\ding{53}} & \checkmark  \\ \hline 
\ref{exSHcont}, & $i\neq 1$ & \small{\ding{53}} & \small{\ding{53}} & \checkmark & \small{\ding{53}} & \checkmark & \small{\ding{53}} & \checkmark  \\ \hline 
\ref{ex1nww} & & \small{\ding{53}} & \small{\ding{53}} & \checkmark & \checkmark & \checkmark & \small{\ding{53}} & \checkmark  \\ \hline 
\ref{ex2} &  & s\checkmark & \small{\ding{53}} & \small{\ding{53}} & \small{\ding{53}} & \small{\ding{53}} & \checkmark & \checkmark  \\ \hline 
\ref{lastex}, & $i\neq 1$ & s\checkmark & \small{\ding{53}} & \small{\ding{53}} & \small{\ding{53}} & \checkmark & \small{\ding{53}} & \checkmark  \\ \hline 
\ref{ex2vi} &  & \checkmark & \small{\ding{53}} & \checkmark & \small{\ding{53}} & \small{\ding{53}} & \checkmark & \checkmark  \\ \hline 
\ref{Vexlast}& & s\checkmark & \small{\ding{53}} & \checkmark & \checkmark & \checkmark & \checkmark & \small{\ding{53}}  \\ \hline 
\ref{Vexlas2}& & \small{\ding{53}} & \small{\ding{53}} & \checkmark & \checkmark & \checkmark & \checkmark & \checkmark  \\ \hline 
\ref{Vex2}& & \checkmark & $V_{j}$ & \checkmark & \small{\ding{53}} & \small{\ding{53}} & \small{\ding{53}} & \checkmark  \\ \hline 
\ref{ex2nww}& & \checkmark & \small{\ding{53}} & \checkmark & \checkmark & \checkmark & \checkmark & \small{\ding{53}}  \\ \hline 
\ref{ex3nww}& & \small{\ding{53}} & $V_{j}$ & \checkmark & \small{\ding{53}} & \checkmark & \checkmark & \small{\ding{53}}  \\ \hline 
\ref{ex6}, & $1\leq i\leq n-2$ & \checkmark & $\ast$ & \checkmark & \small{\ding{53}} & \small{\ding{53}} & \checkmark & \checkmark \\ \hline 
\ref{exoct2}& & \small{\ding{53}} & $V_{j}$ & \checkmark & \checkmark & \checkmark & \small{\ding{53}} & \checkmark \\ \hline 
\ref{exoct}& & \small{\ding{53}} & $V_{j}$ & \checkmark & \small{\ding{53}} & \small{\ding{53}} & \checkmark & \checkmark \\ \hline 
\ref{ex1Gab}& & s\checkmark & \small{\ding{53}} & \checkmark & \checkmark & \checkmark & \checkmark & \small{\ding{53}} \\ \hline 
\end{tabular}
\end{center}
\vspace{.2in}
\caption{Properties, numbered as in Section~\ref{symm}, of the various examples of symmetrizations from Section~\ref{further}, where s\checkmark indicates strictly monotonic and where $\ast$ indicates $V_{n-i}$-preserving when $j=n-i$.}
\label{table2}
\end{table}

\section{Relations between properties}\label{relations}

It was noted above that invariance on $H$-symmetric sets implies invariance on $H$-symmetric spherical cylinders and idempotence.  In this section we establish some less obvious relations between the properties considered in Section~\ref{symm}.

Note that the conclusion in the following result holds trivially when $i=0$.

\begin{thm}\label{Invariance}
Let $i\in\{1,\dots,n-1\}$, let $H\in {\mathcal{G}}(n,i)$, and let ${\mathcal{B}}={\mathcal{K}}^n$ or ${\mathcal{B}}={\mathcal{K}}^n_n$.  Suppose that ${\di}:{\mathcal{B}}\rightarrow{\mathcal{B}}_H$ is monotonic and either

{\rm{(i)}} $i=1$ and $\di$ is invariant on $H$-symmetric spherical cylinders, or

{\rm{(ii)}} $i\in\{2,\dots,n-1\}$ and ${\di}$ is invariant on $H$-symmetric sets.

Then $\di$ is projection invariant.
\end{thm}

\begin{proof}
Let ${\mathcal{B}}={\mathcal{K}}^n$ or ${\mathcal{K}}^n_n$.  Let $K\in {\mathcal{B}}$ and recall that $K\neq\emptyset$ by the definition of ${\mathcal{K}}^n$. Choose $R_0>0$ such that $L=(K|H)+ (R_0B^n\cap H^{\perp})$ contains $K$.  Then $L\in {\mathcal{B}}$ is an $H$-symmetric set and if $i=1$, then $L$ is an $H$-symmetric spherical cylinder.  Our assumptions in (i) and (ii) imply that $\di L=L$.  The monotonicity of $\di$ yields $\di K\subset \di L$, so
$$(\di K)|H\subset (\di L)|H=L|H=K|H.$$

To prove the opposite containment, suppose to the contrary that there is a ball $B$  with $\dim B=\dim(K|H)$ such that $B\subset (K|H)\setminus ((\di K)|H)$.  Choose $R_1>0$ such that if $M=B+(R_1B^n\cap H^{\perp})$, then $K\cap (B\times H^{\perp})\subset M$.  Since $K\cap M\in {\mathcal{B}}$, $M\in {\mathcal{B}}$, and $M^{\dagger}=M$, the monotonicity of $\di$ and its invariance on $H$-symmetric spherical cylinders imply that $\di(K\cap M)\subset \di M=M$ and hence $(\di(K\cap M))|H\subset M|H=B$.  From $K\cap M\neq \emptyset$ we conclude that $\di(K\cap M)$ is defined and nonempty, and thus $(\di(K\cap M))|H\neq \emptyset$.  The inclusion $K\cap M\subset K$ gives $(\di(K\cap M))|H\subset (\di K)|H$.  Consequently, $B\cap ((\di K)|H)\neq \emptyset$.  This contradiction shows that $K|H\subset (\di K)|H$.  Therefore $(\di K)|H=K|H$ and $\di$ is projection invariant.
\end{proof}

The proof of the previous theorem shows that in (ii) it is only necessary to assume that $\di$ is invariant on sets of the form $(K|H)+s(B^n\cap H^{\perp})$, where $K\in {\mathcal{B}}$ and $s>0$.  The latter assumption is still stronger than invariance on $H$-symmetric spherical cylinders, which Example~\ref{ex1nwwmar} shows does not suffice.  Examples~\ref{ex1nww} and \ref{Vex2} show that the other assumptions cannot be omitted and Example~\ref{ex2} shows that the converse of Theorem~\ref{Invariance} is false in the sense that projection invariance does not imply invariance on $H$-symmetric spherical cylinders in the presence of monotonicity.

Next, we focus on the idempotent property.

\begin{lem}\label{Gab1}
Let $H\in {\mathcal{G}}(n,i)$, $i\in\{0,\dots,n-1\}$, and let ${\mathcal{B}}\subset {\mathcal{C}}^n$.  Suppose that ${\di}:{\mathcal{B}}\rightarrow{\mathcal{B}}_H$ is strictly monotonic and idempotent. If $K\in {\mathcal{B}}$ and either $\di K\subset K$ or $\di K\supset K$, then $\di K=K$.
\end{lem}

\begin{proof}
If $K\in {\mathcal{B}}$ and the inclusion $\di K\subset K$ is proper, then so is the inclusion $\di^2 K\subset \di K$, by the strict monotonicity of $\di$.  This contradicts the equality $\di^2K =\di K$ provided by the idempotence of $\di$.  A similar argument applies when the inclusion $\di K\supset K$ holds.
\end{proof}

\begin{thm}\label{Gab3}
Let $H\in {\mathcal{G}}(n,n-1)$ and let ${\mathcal{B}}={\mathcal{K}}^n$ or ${\mathcal{B}}={\mathcal{K}}^n_n$.  Suppose that ${\di}:{\mathcal{B}}\rightarrow{\mathcal{B}}_H$ is strictly monotonic and idempotent.  The following are equivalent:

{\rm{(i)}}  $\di$ is invariant on $H$-symmetric sets;

{\rm{(ii)}}  $\di$ is invariant on $H$-symmetric spherical cylinders;

{\rm{(iii)}}  $\di$ is projection invariant.
\end{thm}

\begin{proof}
The implication (i)$\Rightarrow$(iii) holds by Theorem~\ref{Invariance}.

(iii)$\Rightarrow$(ii) Let $K$ be an $H$-symmetric spherical cylinder.  If $\di K\subset K$ or $\di K\supset K$, then $\di K=K$ by Lemma~\ref{Gab1}.  Otherwise, since $\dim H=n-1$, there is an $x\in \relint K|H=\relint(\di K)|H$ such that $(\di K)\cap (H^{\perp}+x)$ is properly contained in $K\cap (H^{\perp}+x)$.  We may then choose $r>0$ so that $(\di K)\cap C_r(x)$ is properly contained in $K\cap C_r(x)$, where $C_r(x)=D_r(x)\times H^{\perp}$ and $D_r(x)\subset H$ is the $(n-1)$-dimensional ball with center $x$ and radius $r>0$.  The monotonicity of $\di$ implies that $\di(K\cap C_r(x))\subset \di K$ and the projection invariance of $\di$ yields $\di(K\cap C_r(x))\subset C_r(x)$. Therefore $\di(K\cap C_r(x))\subset (\di K)\cap C_r(x)$ and it follows that $\di(K\cap C_r(x))$ is properly contained in $K\cap C_r(x)$. This contradicts Lemma~\ref{Gab1}, applied to the set $K\cap C_r(x)$.

(ii)$\Rightarrow$(i) Assume that $K\in {\mathcal{B}}$ is an $H$-symmetric set and initially also that $\dim K=n$.  If $L$ is an $H$-symmetric spherical cylinder with $L\subset K$, then $L=\di L\subset \di K$.  Since $\dim H=n-1$ and $\dim K=n$, it is easy to see that $K$ is the closure of the union of the $H$-symmetric spherical cylinders contained in it. Consequently, $K\subset \di K$ and it follows from Lemma~\ref{Gab1} that $\di K=K$.

Now assume that $K$ is an $H$-symmetric set and $\dim K<n$. For each $\ee>0$, $K+\ee B^n$ is $H$-symmetric and $\dim(K+\ee B^n)=n$, so $\di K\subset \di(K+\ee B^n)=K+\ee B^n$. Since $K=\cap\{K+\ee B^n:\ee>0\}$, this shows that $\di K\subset K$ and we conclude that $\di K=K$ as before.
\end{proof}

Assuming $i=n-1$, Examples~\ref{ex1nwwmar}, \ref{ex2}, \ref{lastex},  \ref{ex2vi}, and \ref{Vex2} with $j=1$ show that neither (ii)$\Rightarrow$(iii) nor (iii)$\Rightarrow$(ii) holds if any of the assumptions in Theorem~\ref{Gab3} (strict monotonicity, idempotence, and either invariance on $H$-symmetric spherical cylinders or projection invariance) is omitted or strict monotonicity is weakened to monotonicity.  Since (i)$\Rightarrow$(ii) is always true by the definitions and (i)$\Rightarrow$(iii) holds in the presence of monotonicity alone by Theorem~\ref{Invariance}, no other implications need be considered.  That the restriction $\dim H=n-1$ cannot be dropped in (iii)$\Rightarrow$(ii) is shown by Example~\ref{ex6} with $j=1$ and $L\neq B^n\cap H^{\perp}$, while outer rotational symmetrization serves the same purpose for (iii)$\Rightarrow$(i) and (ii)$\Rightarrow$(i).  We do not have an example to show that the assumption $\dim H=n-1$ is needed in (ii)$\Rightarrow$(iii); see Problem~\ref{prob1}.

The next lemma will allow us to prove Theorem~\ref{SchwarzFpre}, a variant of Theorem~\ref{Gab3} that achieves the same conclusion with different assumptions.

\begin{lem}\label{SchwarzFGab}
Let $H\in {\mathcal{G}}(n,i)$, $i\in\{0,\dots,n-1\}$, let ${\mathcal{B}}={\mathcal{K}}^n$ or ${\mathcal{B}}={\mathcal{K}}^n_n$, and let $F:{\mathcal{B}}\to [0,\infty)$ be a strictly increasing set function invariant under translations orthogonal to $H$ of $H$-symmetric sets.  Suppose that ${\di}:{\mathcal{B}}\rightarrow{\mathcal{B}}_H$ is monotonic and $F$-preserving.
If $K\in {\mathcal{B}}$ is $H$-symmetric and either $\di(K+y)\subset K$ or $\di(K+y)\supset K$ for some $y\in H^{\perp}$, then $\di(K+y)=K$.
\end{lem}

\begin{proof}
If $K\in {\mathcal{B}}$ is $H$-symmetric, $y\in H^{\perp}$, and $\di(K+y)$ is a proper subset of $K$, then since $F$ is strictly increasing, $F(\di(K+y))< F(K)$. But $F(\di(K+y))=F(K+y)=F(K)$, because $\di$ is $F$-preserving and $F$ is invariant under translations orthogonal to $H$ of $H$-symmetric sets.  A similar contradiction is reached if $K$ is a proper subset of $\di(K+y)$.
\end{proof}

\begin{thm}\label{SchwarzFpre}
Let $H\in {\mathcal{G}}(n,n-1)$, let ${\mathcal{B}}={\mathcal{K}}^n$ or ${\mathcal{B}}={\mathcal{K}}^n_n$, and let $F:{\mathcal{B}}\to [0,\infty)$ be a strictly increasing set function. If ${\di}:{\mathcal{B}}\rightarrow{\mathcal{B}}_H$ is monotonic and $F$-preserving, the following are equivalent:

{\rm{(i)}}  $\di$ is invariant on $H$-symmetric sets;

{\rm{(ii)}}  $\di$ is invariant on $H$-symmetric spherical cylinders;

{\rm{(iii)}}  $\di$ is projection invariant.
\end{thm}

\begin{proof}
Note firstly that when $y=o$, the assumption that $F$ is invariant under translations orthogonal to $H$ of $H$-symmetric sets is not needed in Lemma~\ref{SchwarzFGab}.  Therefore to prove the equivalence of the conditions {\rm{(i)}}--{\rm{(iii)}}, we can follow the proof of Theorem~\ref{Gab3}, using Lemma~\ref{SchwarzFGab} with $y=o$ to replace Lemma~\ref{Gab1}, since this proof does not otherwise require the strict monotonicity or idempotence of $\di$.
\end{proof}

Assuming $i=n-1$ and taking $F=V_1$, Examples~\ref{ex2}, \ref{lastex}, \ref{Vex2} with $j=1$, and \ref{exoct}, together with Blaschke symmetrization with $F=V_{n-1}$, show that neither (ii)$\Rightarrow$(iii) nor (iii)$\Rightarrow$(ii) holds if any of the three other assumptions is omitted or strict monotonicity is weakened to monotonicity.  Examples~\ref{ex1nwwmar} and \ref{ex2vi} with $F(K)=V_1(\di K)$ show that it does not suffice to assume that $F$ is merely increasing. That the restriction $\dim H=n-1$ cannot be dropped in (iii)$\Rightarrow$(ii) is shown by Example~\ref{ex6} with $F=V_{n-i}$, $j=1$, and $L\neq B^n\cap H^{\perp}$, while Schwarz symmetrization with $F=V_n$ serves the same purpose for (iii)$\Rightarrow$(i) and (ii)$\Rightarrow$(i).  We do not have an example to show that the assumption $\dim H=n-1$ is needed in (ii)$\Rightarrow$(iii); see Problem~\ref{prob2}.

\section{The role of Steiner and Minkowski symmetrization}\label{idemp}

We now present some expressions for the Steiner (or, more generally, fiber) and Minkowski symmetrals that shed light on the relationship between them and which will find use in the sequel.  Recall that $F_HK=\triangle K=M_HK$ if $i=0$ and that $F_HK=S_HK$ if $i=n-1$, so the formula (\ref{Steinnew}) below provides, in particular, an expression for the Steiner symmetral.

\begin{thm}\label{Newchar}
Let $H\in {\mathcal{G}}(n,i)$, $i\in\{0,\dots,n-1\}$, and for $K\in {\mathcal{K}}^n$ and $y\in H^{\perp}$, let
\begin{equation}\label{ktd}
K_y=K+y~~\quad~~{\rm{and}}~~\quad~~K_y^{\dagger}=(K_y)^{\dagger}=K^{\dagger}-y.
\end{equation}
Then for $K\in {\mathcal{K}}^n$, we have
\begin{equation}\label{Steinnew}
F_HK=\bigcup_{y\in H^{\perp}}\,(K_y\cap K^{\dagger}_y)
\end{equation}
and
\begin{equation}\label{Minnew}
M_HK=\bigcap_{y\in H^{\perp}}\,\conv(K_y\cup K^{\dagger}_y).
\end{equation}
\end{thm}

\begin{proof}
Let $z\in F_HK$.  Then, using (\ref{fhjk}) with $G=H$, we have $z=((x+a)/2)+(x-b)/2$, where $x\in H$, $a,b\in H^{\perp}$, and $x+a,x+b\in K$.  Let $y=-(a+b)/2$.  Then $z=(x+a)+y\in K+y$ and $z=(x-b)-y\in K^{\dagger}-y$.  Therefore
$z\in K_y\cap K^{\dagger}_y$, so $F_HK$ is contained in the right-hand side of (\ref{Steinnew}).  For the reverse inclusion, note first that $K_y\cap K^{\dagger}_y$ is $H$-symmetric.  From the invariance of $F_H$ on $H$-symmetric sets and the fact that $F_H$ is monotonic and invariant on translations orthogonal to $H$ of $H$-symmetric sets, we obtain
$$K_y\cap K^{\dagger}_y=F_H(K_y\cap K^{\dagger}_y)\subset F_HK_y=F_HK$$
for all $y\in H^{\perp}$.  This proves (\ref{Steinnew}).

To prove (\ref{Minnew}), let $y\in H^{\perp}$ and let $Q_y=\conv(K_y\cup K^{\dagger}_y)$.  Then since $K_y$ and $K^{\dagger}_y$ are contained in $Q_y$ and the latter is convex,
$$M_HK=M_HK_y=\frac12 K_y+\frac12 K^{\dagger}_y\subset Q_y,$$
so $M_HK\subset \cap_{y\in H^{\perp}}Q_y$. To prove the reverse containment in (\ref{Minnew}), observe that if $v\in S^{n-1}$, then by (\ref{ktd}) and \cite[(0.24), p.~17]{Gar06}, we obtain
\begin{eqnarray}\label{n7}
h_{\cap_{y\in H^{\perp}}\,Q_y}(v)&\le &\min_{y\in H^{\perp}}h_{Q_y}(v)=
\min_{y\in H^{\perp}}\max\{h_{K_y}(v),h_{K^{\dagger}_y}(v)\}\nonumber\\
&=&\min_{y\in H^{\perp}}\max\left\{h_K(v)+y\cdot v,h_{K^{\dagger}}(v)-y\cdot v\right\}\nonumber\\
&=&\frac12 h_{K}(v)+\frac12 h_{K^{\dagger}}(v)=h_{M_HK}(v),
\end{eqnarray}
as required, where the first equality in (\ref{n7}) results from observing that the minimum occurs when the two expressions are equal, i.e., when $y\cdot v=(h_{K^{\dagger}}(v)-h_K(v))/2$.
\end{proof}

\begin{cor}\label{Newcharcor}
If $H\in {\mathcal{G}}(n,i)$, $i\in\{0,\dots,n-1\}$, and $K\in {\mathcal{K}}^n$, then the fiber symmetral $F_HK$ (and therefore the Steiner symmetral $S_HK$, if $i=n-1$) is the union of all $H$-symmetric compact convex sets such that some translate orthogonal to $H$ is contained in $K$, and the Minkowski symmetral $M_HK$ is the intersection of all $H$-symmetric compact convex sets such that some translate orthogonal to $H$ contains $K$.
\end{cor}

\begin{proof}
For each $y\in H^{\perp}$, the sets $K_y\cap K_y^{\dagger}$ and $\conv(K_y\cup K_y^{\dagger})$ defined via (\ref{ktd}) are $H$-symmetric. It remains to observe that if $L\in {\mathcal{K}}^n$ is $H$-symmetric and $L+z\subset K$ (or $L+z\supset K$) for some $z\in H^{\perp}$, then $L\subset K_{-z}\cap K^{\dagger}_{-z}$ (or $L\supset \conv(K_{-z}\cup K^{\dagger}_{-z})$, respectively).
\end{proof}

By Theorem~\ref{Gab3}, when $i=n-1$ the hypotheses of the following corollary hold if the assumption that $\di$ is monotonic and invariant under $H$-symmetric sets is replaced by the assumption that $\di$ is strictly monotonic, idempotent, and either invariant on $H$-symmetric spherical cylinders or projection invariant.

\begin{cor}\label{IdemGab}
Let $H\in {\mathcal{G}}(n,i)$, $i\in\{0,\dots,n-1\}$, and let ${\mathcal{B}}={\mathcal{K}}^n$ or ${\mathcal{B}}={\mathcal{K}}^n_n$.  Suppose that ${\di}:{\mathcal{B}}\rightarrow{\mathcal{B}}_H$ is monotonic, invariant on $H$-symmetric sets, and invariant under translations orthogonal to $H$ of $H$-symmetric sets.  Then
\begin{equation}\label{SCM2a}
F_HK\subset\di K\subset M_HK
\end{equation}
for all $K\in {\mathcal{B}}$.
\end{cor}

\begin{proof}
Let $K\in {\mathcal{B}}$ and let $y\in H^{\perp}$. The sets $K_y\cap K_y^{\dagger}$ and $\conv(K_y\cup K_y^{\dagger})$ defined via (\ref{ktd}) are $H$-symmetric and $K\subset \conv(K_y\cup K_y^{\dagger})-y$.  Hence, using the monotonicity and invariance property of $\di$, we obtain
\begin{equation}\label{GG1}
K_y\cap K_y^{\dagger}=\di(K_y\cap K_y^{\dagger})\subset \di K_y=\di K
\end{equation}
and
\begin{equation}\label{GG2}
\di K\subset \di\left(\conv(K_y\cup K_y^{\dagger})-y\right)=\di\conv(K_y\cup K_y^{\dagger})=\conv(K_y\cup K_y^{\dagger}).
\end{equation}
Then (\ref{SCM2a}) follows immediately from (\ref{Steinnew}), (\ref{Minnew}), (\ref{GG1}), and (\ref{GG2}).
\end{proof}

The containment $S_HK\subset M_HK$ when $i=n-1$ is both well known and rather obvious geometrically, but nevertheless has been found useful in proving results on the convergence of successive Steiner symmetrals; see, for example, \cite{CouD14} and the references given there.

The restriction to $i=n-1$ cannot be dropped in obtaining the inclusion $S_HK\subset\di K$, since it is not true that for $i\in \{1,\dots,n-2\}$ and $H\in {\mathcal{G}}(n,i)$, we have $S_HK\subset M_HK$, where $S_HK$ is the Schwarz symmetral of $K$.  Indeed, if $H=\lin\{e_1,\dots,e_i\}$, $a>0$, and $K=[-1,1]^{n-1}\times[-a,a]$, then $M_HK=K$ and $S_HK$ is a spherical cylinder with radius $r$ satisfying $\kappa_{n-i}r^{n-i}=2^{n-i}a$.  Thus if $a<(2^{n-i}/\kappa_{n-i})^{1/(n-i-1)}$, then $r>a$ and $S_HK\not\subset M_HK$.

Examples~\ref{Vexlast}, \ref{Vexlas2}, and \ref{Vex2} with $j=1$ show that none of the other assumptions can be omitted in obtaining the inclusion $F_HK\subset\di K$. Examples~\ref{ex1nww}, \ref{ex2}, and \ref{ex1Gab} show that none of the assumptions can be omitted in obtaining the inclusion $\di K\subset M_HK$.

The fiber symmetrizations $F_{H,G}$ satisfy the hypotheses of Corollary~\ref{IdemGab}.   Another such family of symmetrizations is given by $\di_{t,J} K=((1-t)\circ S_HK)\nplus_J (t\circ M_HK)$ for $K\in {{\mathcal{K}}^n}$ and $0\le t\le 1$, where $\dim J\in \{0,\dots,n\}$ and $J\subset H$ or $J\supset H$.  Further examples can be obtained by concatenating these symmetrizations for different $t$ and $J$.

\begin{thm}\label{SchwarzF}
Let $H\in {\mathcal{G}}(n,n-1)$, let ${\mathcal{B}}={\mathcal{K}}^n$ or ${\mathcal{B}}={\mathcal{K}}^n_n$, and let $F:{\mathcal{B}}\to [0,\infty)$ be a strictly increasing set function invariant under translations orthogonal to $H$ of $H$-symmetric sets.  If ${\di}:{\mathcal{B}}\rightarrow{\mathcal{B}}_H$ is monotonic, $F$-preserving, and either invariant on $H$-symmetric spherical cylinders or projection invariant, then $\di$ is invariant under translations orthogonal to $H$ of $H$-symmetric sets and
\begin{equation}\label{SCM}
S_HK\subset \di K\subset M_HK
\end{equation}
for all $K\in {\mathcal{B}}$.
\end{thm}

\begin{proof}
By Theorem~\ref{SchwarzFpre}, $\di$ is invariant on $H$-symmetric sets.  Then (\ref{SCM}) will follow from Corollary~\ref{IdemGab} if we can show that $\di$ is invariant under translations orthogonal to $H$ of $H$-symmetric sets.  To this end, let $K\in {\mathcal{B}}$ be $H$-symmetric and let $y\in H^{\perp}$.  The desired conclusion that $\di (K+y)=\di K$ will follow if we show that $\di(K+y)=K$.

Suppose that $\di(K+y)\neq K$.  By Theorem~\ref{Invariance}, $\di$ is projection invariant.  Assuming $\dim K=n$, it follows from Lemma~\ref{SchwarzFGab} that there is an $x\in \relint K|H=\relint(\di K)|H$ such that $(\di(K+y))\cap (H^{\perp}+x)$ is properly contained in $K\cap (H^{\perp}+x)$.  We may then choose $r>0$ so that $(\di(K+y))\cap C_r(x)$ is properly contained in $K\cap C_r(x)$, where $C_r(x)=D_r(x)\times H^{\perp}$ and $D_r(x)\subset H$ is the $(n-1)$-dimensional ball with center $x$ and radius $r>0$.  By the monotonicity and projection invariance of $\di$, we have
$$\di((K\cap C_r(x))+y)=\di((K+y)\cap C_r(x))\subset (\di(K+y))\cap C_r(x),$$
and hence $\di((K\cap C_r(x))+y)$ is properly contained in $K\cap C_r(x)$.  This contradicts Lemma~\ref{SchwarzFGab}, applied to the set $K\cap C_r(x)$.

Now suppose that $\dim K<n$.  By the monotonicity of $\di$ and the above, for each $\ee>0$ we have
$$\di(K+y)\subset \di(K+\ee B^n+y)=\di(K+\ee B^n)=K+\ee B^n,$$
which implies that $\di(K+y)\subset K$ and hence, by Lemma~\ref{SchwarzFGab} again, that $\di(K+y)=\di K$.
\end{proof}

The restriction to $i=n-1$ cannot be dropped in obtaining either inclusion in (\ref{SCM}). Indeed, we may take $\di=S_H$, Schwarz symmetrization, and $F=V_n$, but we showed after Corollary~\ref{IdemGab} that for $i\in \{1,\dots,n-2\}$ and $H\in {\mathcal{G}}(n,i)$, it is not generally true that $S_HK\subset M_HK$.  A similar conclusion is reached by taking $\di=\overline{M}_H$, Minkowski-Blaschke symmetrization, and $F=V_1$.

Examples~\ref{Vex2}, \ref{ex2nww}(i) and (ii), and \ref{ex3nww} with $j=1$ or $j=n$ show that none of the other assumptions can be omitted in obtaining the inclusions in (\ref{SCM}).

\begin{thm}\label{IdemGab2}
Let $H\in {\mathcal{G}}(n,i)$, $i\in\{1,\dots,n-1\}$, let ${\mathcal{B}}={\mathcal{K}}^n$ or ${\mathcal{B}}={\mathcal{K}}^n_n$, and let
${\di}:{\mathcal{B}}\rightarrow{\mathcal{B}}_H$ be invariant on $H$-symmetric spherical cylinders and invariant under translations orthogonal to $H$ of $H$-symmetric sets.  Consider the expression
\begin{equation}\label{SCM2c}
I_HK\subset \di K\subset O_HK.
\end{equation}
The left-hand inclusion holds for all $K\in {\mathcal{B}}$ if, in addition to the assumptions stated before \eqref{SCM2c}, ${\mathcal{B}}={\mathcal{K}}^n_n$ and $\di$ is monotonic.  The right-hand inclusion holds for all $K\in {\mathcal{B}}$ if, in addition to the assumptions stated before \eqref{SCM2c}, $\di$ is strictly monotonic and idempotent.
\end{thm}

\begin{proof}
Suppose that $\di$ is invariant on $H$-symmetric spherical cylinders and invariant under translations orthogonal to $H$ of $H$-symmetric sets.  Let $K\in {\mathcal{K}}^n_n$. Assume $\di$ is also monotonic. Let $L$ be an $H$-symmetric spherical cylinder such that $L+y\subset K$ for some $y\in H^{\perp}$.  Then
$$L=\di L=\di(L+y)\subset \di K.$$
Since $I_HK$ is the closure of the union of such sets $L$, the left-hand inclusion in (\ref{SCM2c}) follows.

Now assume instead that $\di$ is also strictly monotonic and idempotent. Let $K\in {\mathcal{K}}^n$ and let $C$ be a rotationally invariant convex body such that $K\subset C+y$ for some $y\in H^{\perp}$.  If $L$ is an $H$-symmetric spherical cylinder with $L\subset C$, then
$L=\di L\subset \di C$.  Since $C$ is the closure of the union of all such sets $L$, we have $C\subset \di C$ and hence $\di C=C$ by Lemma~\ref{Gab1}.  Therefore
$$\di K\subset \di(C+y)=\di C=C.$$
The right-hand inclusion in (\ref{SCM2c}) now follows from the definition of $O_HK$.
\end{proof}

The assumption ${\mathcal{B}}={\mathcal{K}}^n_n$ cannot be omitted in obtaining the inclusion $I_HK\subset\di K$, as can be seen by defining ${\di}:{\mathcal{K}}^n\rightarrow{\mathcal{K}}^n_H$ by $\di K=I_HK$ if $\dim K=n$ and $\di K=K|H$ otherwise.  Examples~\ref{exSHcont}, \ref{Vexlast} with ${\mathcal{B}}={\mathcal{K}}^n_n$, and \ref{Vex2} with ${\mathcal{B}}={\mathcal{K}}^n_n$ and $j=1$ show that none of the other assumptions can be omitted either.  Examples~\ref{ex1nwwmar}, \ref{lastex}, \ref{Vex2} with $j=1$, and \ref{ex1Gab} serve the same purpose for the inclusion $\di K\subset O_HK$.

Note that we may take $\di=S_H$, Schwarz symmetrization, in the previous theorem. Other symmetrizations satisfying the hypotheses of Theorem~\ref{IdemGab2} are the fiber symmetrizations $F_{H,G}$ or those given by $\di_{t,J}K=((1-t)\circ I_HK)\nplus_J (t\circ O_HK)$ for $K\in {{\mathcal{K}}^n}$ and $0\le t\le 1$, where $\dim J\in \{0,\dots,n\}$ and $J\subset H$ or $J\supset H$, and again, further examples can be obtained by concatenating these symmetrizations for different $t$ and $J$.

In the context of Corollary~\ref{IdemGab} and Theorem~\ref{IdemGab2} it is natural also to consider defining $\di_t K=(1-t)S_HK+_ptM_HK$ or $\di_t K=(1-t)I_HK+_ptO_HK$, where $1\le p\le \infty$ and the $L_p$ sum is taken with respect to the centroid of $S_HK$ or $I_HK$, as appropriate.  However, $\di_t$ is not monotonic when $p>1$.

\section{Convergence of successive symmetrals}\label{Convergence}

Elsewhere in this paper we always consider a symmetrization ${\di}=\di_H:{\mathcal{B}}\rightarrow{\mathcal{B}}_H$ with respect to a fixed subspace $H$.  In this section, we fix $i\in \{1,\dots,n-1\}$ but regard $\di$ as the entire collection of such maps, for all $H\in{\mathcal{G}}(n,i)$, and consider the convergence of successive applications of $\di$ through a sequence of $i$-dimensional subspaces. (We do not attempt to obtain optimal results; the topic will be thoroughly investigated in a future paper.)  To keep the terminology clear, we refer to the collection $\di$ of maps $\di_H$ as a {\em symmetrization process}.  We shall use and extend ideas of Coupier and Davydov \cite{CouD14}, who consider only the case $i=n-1$, and adopt some of their notation in modified form.

Let $i\in \{1,\dots,n-1\}$ and suppose that $\di$ is a symmetrization process such that for each $H\in{\mathcal{G}}(n,i)$,  ${\di_H}:{\mathcal{B}}\rightarrow{\mathcal{B}}_H$ is an $i$-symmetrization.  Let $(H_m)$ be a sequence in ${\mathcal{G}}(n,i)$ and for convenience write $\di_m=\di_{H_m}$ for $m\in\N$. If $1\le j\le m$, let
$$\di_{j,m}K=\di_{m}(\di_{{m-1}}(\cdots(\di_jK)\cdots))$$
for each $K\in {\mathcal{B}}$, so that $\di_{j,m}K$ results from $m-j+1$ successive $\di$-symmetrizations applied to $K$ with respect to $H_j,H_{j+1},\dots,H_m$.

Let ${\mathcal{B}}={\mathcal{K}}_n^n$. A sequence $(H_m)$ in ${\mathcal{G}}(n,i)$ is called {\em weakly $\di$-universal} if for all $K\in {\mathcal{K}}_n^n$ and $j\in \N$, there exists $r(j,K)>0$ such that $\di_{j,m}K\to r(j,K)B^n$ as $m\to \infty$.  Note that this implies in particular that the successive symmetrals $\di_{1,m}K$ converge to a ball as $m\to \infty$.  If the constant $r(j,K)$ is independent of $j$, we say that $(H_m)$ is {\em $\di$-universal}.  Example~\ref{weakdi} below exhibits a symmetrization $\di$ and a sequence $(H_m)$ that is weakly $\di$-universal but not $\di$-universal.

We shall use the terms {\em (weakly) $S$-universal} or {\em (weakly) $M$-universal} when $\di$ is Steiner or Minkowski symmetrization, respectively, and when $i=n-1$, also use the terms for the sequences $(u_m)$ of directions in $S^{n-1}$ such that $H_m=u_m^{\perp}$ for each $m$.  In fact, by \cite[Theorem~3.1]{CouD14}, a sequence $(u_m)$ in $S^{n-1}$ is $S$-universal if and only if it is $M$-universal.   Since Steiner and Minkowski symmetrization preserve volume and mean width, respectively, it is easy to see that $(u_m)$ is weakly $S$-universal (or weakly $M$-universal) if and only if it is $S$-universal (or $M$-universal, respectively).  Much is known about such sequences; in particular, \cite[Proposition~3.3]{CouD14} implies that any countable dense subset in $S^{n-1}$ contains a sequence that has each of these four equivalent properties.

The proof of the following result essentially follows that of \cite[Theorem~3.1]{CouD14} and we include it for the convenience of the reader.

\begin{thm}\label{Successive}
Let $i\in \{1,\dots,n-1\}$ and let $\di$ be a symmetrization process on ${\mathcal{K}}^n_n$.  Suppose that
\begin{equation}\label{SdiM}
I_HK\subset\di_HK\subset M_HK
\end{equation}
for all $H\in {\mathcal{G}}(n,i)$ and all $K\in {\mathcal{K}}^n_n$.  If $(H_m)$ is an $M$-universal sequence in ${\mathcal{G}}(n,i)$, then $(H_m)$ is weakly $\di$-universal.
\end{thm}

\begin{proof}  Let $(H_m)$ be $M$-universal and let $K\in {\mathcal{K}}^n_n$. It is easy to see that any ball with center at the origin that contains $K$ will also contain all the successive $\di$-symmetrals $\di_{1,m}K$. If $m\in \N$ and $L\in {\mathcal{K}}^n_n$, then by (\ref{SdiM}), we have $\di_mL\subset M_mL$ and hence $V_1(\di_mL)\le V_1(M_mL)=V_1(L)$.  Taking $L=\di_{1,m-1}K$, we obtain $V_1(\di_{1,m}K)\le V_1(\di_{1,m-1}K)$ for all $m=2,3,\dots$.  Therefore $V_1(\di_{1,m}K)\to a$, say, as $m\to \infty$.

Since $K\in {\mathcal{K}}^n_n$, there is an $n$-dimensional ball contained in $K$ of radius $b>0$. For any $L\in {\mathcal{K}}^n_n$ and $m\in \N$, let $I_mL=I_{H_m}L$.  From the definition of the inner rotational symmetral, it is clear that if $L$ contains an $n$-dimensional ball of radius $b$, then $I_mL$ does as well.  Then (\ref{SdiM}) implies that $\di_{1,m}K$, $m\in N$, also contains an $n$-dimensional ball of radius $b$.  It follows that $a>0$. By Blaschke's selection theorem, there is a subsequence $(H_{m_p})$ of $(H_m)$ such that $\di_{1,m_p}K\to J\in{\mathcal{K}}^n_n$ as $p\to\infty$, where $V_1(J)=a>0$.

Now if $1\le p < s$, then by (\ref{SdiM}),
\begin{equation}\label{Jin}
\di_{1,m_s}K=\di_{m_p+1,m_s}(\di_{1,m_p}K)\subset M_{m_p+1,m_s}(\di_{1,m_p}K).
\end{equation}
As $s\to\infty$, the body on the left converges to $J$, while because $(H_m)$ is $M$-universal, the body on the right converges to the ball $B_{p,K}$ with center at the origin such that $V_1(B_{p,K})=V_1(\di_{1,m_p}K)$.  However, the latter equation implies that $V_1(B_{p,K})\to a$ as $p\to\infty$.  Since $V_1$ is strictly monotonic, $J\subset B_{p,K}$ by (\ref{Jin}), and $V_1(J)=a$, this forces $J$ to be the ball $B_1$ centered at the origin with $V_1(B_1)=a$.  Consequently, any convergent subsequence of $(\di_{1,m}K)$ converges to $B_1$ and hence $\di_{1,m}K\to B_1$ as $m\to \infty$.

Finally, if $j\in \N$, $j\ge 2$, we can apply the above argument to the $M$-universal sequence $(H_{m+j-1})$, $m\in \N$, to conclude that $\di_{j,m}K$ converges to a ball $B_j$ as $m\to \infty$. This proves that $(H_m)$ is weakly $\di$-universal.
\end{proof}

Note that if $i=n-1$, then $I_H=S_H$, the Steiner symmetral.  As noted above, $M$-universal and $S$-universal sequences coincide.  By Theorem~\ref{Gab3}, when $i=n-1$ the hypotheses of the following corollary hold if the assumption that $\di$ is monotonic and invariant under $H$-symmetric sets is replaced by the assumption that $\di$ is strictly monotonic, idempotent, and either invariant on $H$-symmetric spherical cylinders or projection invariant.

\begin{cor}\label{Successivecor}
Let $i\in \{1,\dots,n-1\}$ and let $\di$ be a symmetrization process on ${\mathcal{K}}^n_n$.  Suppose that for each $H\in {\mathcal{G}}(n,i)$,
${\di}:{\mathcal{K}}^n_n\rightarrow ({\mathcal{K}}^n_n)_H$ is monotonic, invariant on $H$-symmetric sets, and invariant under translations orthogonal to $H$ of $H$-symmetric sets.  If $(H_m)$ is an $M$-universal sequence, then $(H_m)$ is weakly $\di$-universal.
\end{cor}

\begin{proof}
By Corollary~\ref{IdemGab} and Theorem~\ref{IdemGab2}, (\ref{SdiM}) holds
for all $H\in {\mathcal{G}}(n,i)$ and all $K\in {\mathcal{K}}^n_n$.  The result follows directly from Theorem~\ref{Successive}.
\end{proof}

Examples~\ref{Vex2} and~\ref{exoct2}, both with $j=1$ (say) and $B^n$ replaced by an $H$-symmetric $n$-dimensional cube, show that the assumptions of invariance on $H$-symmetric sets and monotonicity, respectively, cannot be dropped in the previous corollary.  We do not have an example showing that the assumption of invariance under translations orthogonal to $H$ of $H$-symmetric sets is necessary.  However, the following example shows that if this assumption is omitted, the hypotheses of Corollary~\ref{Successivecor} do not allow the stronger conclusion that $(H_m)$ is $\di$-universal.

\begin{ex}\label{weakdi}
{\em  Let $\di$ be the symmetrization process corresponding to the symmetrization $\di_H$ in Example~\ref{Vexlast}, with $n=2$ and $i=1$.
Let $0<\theta<\pi/2$ be an irrational multiple of $\pi$ and let $H_1$ be the line through the origin in the direction $(\cos\theta,\sin\theta)$.
For $m\in \N$, let $H_{2m+1}=H_1$ and $H_{2m}=e_2^{\perp}$.  Then the sequence $(H_m)$ is $S$-universal; see, for example, \cite[Corollary~5.4]{Kla12}.

We claim that $(H_m)$ is weakly $\di$-universal.  (Note that this is not a consequence of Corollary~\ref{Successivecor}.) To see this, let $K$ be a planar convex body, let $j\in \N$, and let $D\subset \di_jK$ be a disk of radius $r>0$ whose center $x_j\in H_j$ is at distance $\|x_j\|$ from the origin and at distance $d\ge 0$, say, from $H_{j+1}$.  Then the distance from $D$ to $H_{j+1}$ is no larger than $d$ and the definition of $\di$ implies that $\di_{j+1}D$ is an ellipse with center $x_{j+1}=x_j|H_{j+1}$ at distance $\|x_j\|\cos\theta$ from the origin and area at least $e^{-d}\pi r^2$.  The distance from $x_{j+1}$ to $H_{j+2}$ is $d\cos\theta$, so similarly, $\di_{j,j+2}D$ is an ellipse with center $x_{j+2}=x_{j+1}|H_{j+2}$ at distance $\|x_j\|\cos^2\theta$ from the origin and area at least $e^{-d\cos\theta}e^{-d}\pi r^2$.  Arguing inductively, we see that for $m\ge j+1$, $\di_{j,m}D$ is an ellipse with center $x_{m}\in H_{m}$ at distance $\|x_j\|\cos^{m-j}\theta$ from the origin and area at least $$\big(\prod_{j=0}^{m-j-1}e^{-d\cos^j\theta}\big)\pi r^2=e^{-d\sum_{j=0}^{m-j-1}\cos^j\theta}\pi r^2\ge e^{-d\sum_{j=0}^{\infty}\cos^j\theta}\pi r^2=e^{-d/(1-\cos\theta)}\pi r^2=c,$$
say.  Now $\di_{j,m}D\subset S_{j,m}D$, which is a disk of radius $r$, so the diameter of $\di_{j,m}D$ is no larger than $2r$. Then the $H_{m}$-symmetric ellipse $\di_{j,m}D$ is contained in a rectangle of width $V_1((\di_{j,m}D)|H_{m})=V_1((\di_{j,m}D)\cap H_{m})$ and height $2r$, so the lower bound for the area of $\di_{j,m}D$ gives
$$V_1((\di_{j,m}D)\cap H_{m})\ge c/(2r)$$
for each $m\ge j+1$.  The distance from the center $x_{m}$ of $\di_{j,m}D$ to the origin is $\|x_j\|\cos^{m-j}\theta\to 0$ as $m\to \infty$, so there is an $m_0=m_0(j)$ such that $o\in \di_{j,m_0}D\subset \di_{j,m_0}K$.  Then for $m\ge m_0+1$, we have
$\di_{j,m}K=\di_{m_0+1,m}(\di_{j,m_0}K)=S_{m_0+1,m}(\di_{j,m_0}K)$. Since $(H_m)$ is $S$-universal, $\di_{j,m}K\to r(j,K)B^n$ for some $r(j,K)>0$, proving the claim.

Now let $K=[-a,a]\times [1,1+b]$, where $a,b>0$ are chosen so that $K\cap H_1\neq\emptyset$ and $o\in K|H_1$.  Then $o\in\di_1K=S_1K$, from which it follows that $\di_{1,m}K=S_{1,m}K$ and hence that $V_2(\di_{1,m}K)=2ab$ for $m\in \N$.  On the other hand, $\di_2K=[-a,a]\times [-b/(2e),b/(2e)]$, so $o\in \di_{2,m}K=S_{2,m}(\di_2K)$ yields $V_2(\di_{2,m}K)=2ab/e$ for $m\ge 2$.  It follows that $r(1,K)\neq r(2,K)$, proving that $(H_m)$ is not $\di$-universal.
\qed}
\end{ex}

\section{Characterizations of Minkowski symmetrization}\label{MinS}

We begin with the following basic characterization.

\begin{thm}\label{MinTrans}
Let $H\in {\mathcal{G}}(n,i)$, $i\in\{0,\dots,n-1\}$.  If ${\di}:{\mathcal{K}}^n\rightarrow{\mathcal{K}}^n_H$ is monotonic, invariant on $H$-symmetric sets, and linear (i.e., $\di(K+L)=\di K+\di L$ for all $K, L\in {\mathcal{K}}^n$), then $\di$ is Minkowski symmetrization with respect to $H$.
\end{thm}

\begin{proof}
If $y\in H^{\perp}$, then
$$\di\{y\}+\di\{y^{\dagger}\}=\di(\{y\}+\{y^{\dagger}\})=\di\{o\}=\{o\},$$
since $\{o\}$ is an $H$-symmetric set. This implies that both $\di\{y\}$ and $\di\{y^\dagger\}$ are singletons. Moreover, by the monotonicity,  both are contained in $\di [y,y^\dag]=[y,y^\dagger]$. Thus $\di\{y\}$ is an $H$-symmetric singleton contained in $[y,y^\dagger]$ and it follows that $\di\{y\}=\{o\}$.  If $K\in {\mathcal{K}}^n$, we then have
$$\di(K+y)=\di K+\di\{y\}=\di K,$$
so $\di$ is invariant under translations orthogonal to $H$. By Corollary~\ref{IdemGab}, $\di K\subset M_HK$ and $\di K^{\dagger}\subset M_H(K^{\dagger})$.  Moreover,
$$K+K^\dagger=\di (K+K^\dagger)=\di K+\di K^\dagger \subset M_H K+\di K^\dagger \subset M_H K+ M_H K^\dagger= M_H(K+K^\dagger)=K+K^\dagger,$$
since $K+K^{\dagger}$ is $H$-symmetric. This implies that $\di K+\di K^\dagger = M_H K+\di K^\dagger$ and by the cancelation law for Minkowski addition \cite[p.~139]{Sch93}, we have $\di K=M_H K$.
\end{proof}

Part (iii) of the following result yields a new characterization of central symmetrization, since $M_H=\triangle$ when $i=0$.

\begin{thm}\label{MinChar}
Let $H\in {\mathcal{G}}(n,i)$, $i\in\{0,\dots,n-1\}$, and let ${\mathcal{B}}={\mathcal{K}}^n$ or ${\mathcal{B}}={\mathcal{K}}^n_n$. Suppose that ${\di}:{\mathcal{B}}\rightarrow{\mathcal{B}}_H$ is monotonic.  Assume in addition either that

{\rm{(i)}} $i=n-1$ and $\di$ is mean width preserving and either invariant on $H$-symmetric spherical cylinders or projection invariant, or that

{\rm{(ii)}} $i\in\{1,\dots,n-1\}$ and $\di$ is mean width preserving, invariant on $H$-symmetric sets, and invariant under translations orthogonal to $H$ of $H$-symmetric sets, or that

{\rm{(iii)}} $i=0$ and $\di$ is invariant on $o$-symmetric sets and invariant under translations of $o$-symmetric sets.

Then $\di$ is Minkowski symmetrization with respect to $H$.
\end{thm}

\begin{proof}
Let $K\in {\mathcal{B}}$. For part (i), we can appeal to Theorem~\ref{SchwarzF} with $F=V_1$, the first intrinsic volume, to conclude that $\di K\subset M_HK$.  Then, since both $\di$ and $M_H$ preserve mean width, we obtain
$$V_1(K)=V_1(\di K)\le V_1(M_HK)=V_1(K).$$
Hence $V_1(\di K)=V_1(M_HK)$ and with $\di K\subset M_HK$ we conclude that $\di K=M_HK$.

For part (ii), we may appeal to Corollary~\ref{IdemGab} to get $\di K\subset M_HK$ and then the conclusion follows as before.

Part (iii) is an immediate consequence of Corollary~\ref{IdemGab}, since when $i=0$ we have $F_H=M_H$.
\end{proof}

Minkowski-Blaschke symmetrization shows that the restriction to $i=n-1$ cannot be dropped in (i).  Examples~\ref{Vex2}, \ref{ex2nww}, and \ref{ex3nww}, all with $j=1$, show that none of the other assumptions in (i) can be omitted.  (Note that if $\di$ is $V_1$-preserving, then $\di K\subset M_HK$ implies that $\di K=M_HK$.)

Concerning (ii) and (iii), Examples~\ref{Vex2} and~\ref{exoct2} (which are also valid when $i=0$), both with $j=1$, show that the invariance on $H$-symmetric sets and monotonicity cannot be omitted. For (ii), fiber symmetrization $F_H$ shows that the mean width preserving property cannot be dropped.  For (iii), Example~\ref{Vexlast} (which is also valid when $i=0$) proves that the assumption of invariance under translations orthogonal to $H$ of $H$-symmetric sets is necessary.  We do not have an example to show that the latter assumption is necessary for (ii); see Problem~\ref{prob3}.

\section{Characterizations of Steiner symmetrization}\label{SS}

\begin{thm}\label{SteinerCompact}
$\mathrm{(i)}$ Let $i\in \{1,\dots,n-1\}$ and let $H\in {\mathcal{G}}(n,i)$.  Suppose that $\di:{\mathcal{C}}^n\rightarrow {\mathcal{C}}^n_{H}$ is an $i$-symmetrization that is monotonic, volume preserving, and invariant on $H$-symmetric spherical cylinders.  Then
\begin{equation}\label{mainineq}
{\mathcal{H}}^{n-i}\left((\di K)\cap (H^{\perp}+x)\right)={\mathcal{H}}^{n-i}\left(K\cap (H^{\perp}+x)\right)
\end{equation}
for all $K\in {\mathcal{C}}^n$ and ${\mathcal{H}}^i$-almost all $x\in H$.

$\mathrm{(ii)}$ Suppose that $i=n-1$ and that in addition to the assumptions in $\mathrm{(i)}$, $(\di K)\cap (H^\perp+x)$ is a line segment for ${\mathcal{H}}^{n-1}$-almost all $x\in H$.  Then $\di$ is essentially Steiner symmetrization on ${\mathcal{C}}^n$, in the sense that for all $K\in {\mathcal{C}}^n$, $(\di K)\cap G=(S_HK)\cap G$ for ${\mathcal{H}}^{n-1}$-almost all lines $G$ orthogonal to $H$.
\end{thm}

\begin{proof}
$\mathrm{(i)}$  Fix $K\in {\mathcal{C}}^n$.  For $r>0$, let $D_r(x)$ denote the $i$-dimensional ball in $H$ with center $x$ and radius $r$.  Choose $s>0$ so that $K|H^\perp \subset s(B^n\cap H^{\perp})$.  Let $C_r(x)=D_r(x)+s(B^n\cap H^{\perp})$ and for all $L\in {\mathcal{C}}^n$ and $x\in H$, define
$$
m_{r,L}(x) = {\mathcal{H}}^n(L\cap C_r(x))
$$
and
\begin{equation}\label{mkk}
m_{L}(x) = {\mathcal{H}}^{n-i}(L\cap (H^\perp+x)).
\end{equation}
We first claim that
\begin{equation}\label{mrdi}
m_{r,\di K} \geq m_{r,K}.
\end{equation}
To see this, let $r>0$ and $x\in H$.  By the invariance on $H$-symmetric spherical cylinders, $\di C_r(x) = C_r(x)$. Since $\di$ is monotonic,
\begin{equation}\label{dik}
\di \big(K\cap C_r(x)\big)\subset (\di K)\cap C_r(x).
\end{equation}
Therefore, using the fact that $\di$ is volume preserving, we obtain
$$
m_{r,\di K}(x)={\mathcal{H}}^n((\di K)\cap C_r(x)) \geq {\mathcal{H}}^n(\di (K\cap C_r(x))) = {\mathcal{H}}^n(K\cap C_r(x)) = m_{r,K}(x).
$$
This proves (\ref{mrdi}).  Next, it follows immediately from Lebesgue's differentiation theorem (see e.g.~\cite[Theorem~8.8]{Rud87}) that
\begin{equation}\label{lebthm}
 \lim_{r\to 0} \frac{m_{r,K}(x)}{{\mathcal{H}}^i(D_r(x))} = m_K(x)~\quad~\text{ and }~\quad~\lim_{r\to 0} \frac{m_{r,\di K}(x)}{{\mathcal{H}}^i(D_r(x))} = m_{\di K}(x)
\end{equation}
for ${\mathcal{H}}^i$-almost all $x\in H$.  From (\ref{mkk}) and the volume-preserving property of $\di$, we obtain
$$
\int_H \left(m_{\di K}(x)-m_K(x)\right)\, dx = {\mathcal{H}}^n(\di K)-{\mathcal{H}}^n(K)=0.
$$
By (\ref{mrdi}) and (\ref{lebthm}), the previous integrand is nonnegative for ${\mathcal{H}}^i$-almost all $x\in H$ and therefore vanishes ${\mathcal{H}}^i$-almost everywhere in $H$, yielding (\ref{mainineq}).

$\mathrm{(ii)}$  This follows directly from (\ref{mainineq}) and the definition of $S_HK$.
\end{proof}

We remark that the assumptions in part (i) of the previous theorem can be weakened, since the result remains true if the monotonicity and invariance on $H$-symmetric spherical cylinders only hold modulo sets of zero ${\mathcal{H}}^n$-measure.  In this case the conclusion in part (ii) is also slightly weaker, namely that for all $K\in {\mathcal{C}}^n$, ${\mathcal{H}}^n((\di K)\triangle S_HK)=0$.  The extra assumption made in part (ii) is very strong, but the following example indicates that it may be difficult to weaken it significantly.

\begin{ex}\label{exFir1}
{\em Let $H=e_n^{\perp}$. For each $K\in {\mathcal{C}}^n$, let $\di_1 K=S_H(K\cap (\R^{n-1}\times [-1,1]))$, let
$$\di_2 K=S_H(K\cap (\R^{n-1}\times ((-\infty, -1)\cup(1,\infty)))),$$
and let $\di K=(\di_1 K)\cup (((\di_2K)\cap (\R^{n-1}\times [0,\infty)))+e_n)\cup (((\di_2K)\cap (\R^{n-1}\times (-\infty,0]))-e_n)$.  Since both $\di_1$ and $\di_2$ are monotonic and preserve the volume of the subset of $K$ on which they act, it follows easily that $\di:{\mathcal{C}}^n\to {\mathcal{C}}^n_H$ is monotonic and volume preserving.  It is straightforward to check that $\di$ is also invariant on $H$-symmetric spherical cylinders, but it is essentially different from Steiner symmetrization with respect to $H$.  Note that while the assumption in Theorem~\ref{SteinerCompact}(ii) is false,  $(\di K)\cap (H^{\perp}+x)$ is the union of at most three line segments for all $x\in H$.
\qed}
\end{ex}

In the following, we write ${\mathcal{K}}^n_{nH}$ instead of $({\mathcal{K}}^n_{n})_{H}$ for the class of convex bodies in $\R^n$ that are $H$-symmetric.

\begin{thm}\label{Schwarz}
Let $i\in \{1,\dots,n-1\}$, let $H\in {\mathcal{G}}(n,i)$, and let $\di:{\mathcal{K}}^n_n\rightarrow {\mathcal{K}}^n_{nH}$ be an $i$-symmetrization.  Suppose that $\di$ is monotonic, volume preserving, and projection invariant.  Then
\begin{equation}\label{vni}
V_{n-i}\left((\di K)\cap (H^{\perp}+x)\right)=V_{n-i}\left(K\cap (H^{\perp}+x)\right)
\end{equation}
for all $K\in {{\mathcal{K}}^n_n}$ and $x\in H$.
\end{thm}

\begin{proof}
Let $K\in {\mathcal{K}}^n_n$.  For $r>0$, let $D_r(x)$ denote the $i$-dimensional ball in $H$ with center $x$ and radius $r$.  Choose $s>0$ so that $(K\cup \di K)|H^\perp \subset s(B^n\cap H^{\perp})$.  Let $x\in \relint(K|H)=\relint((\di K)|H)$, let $C_r(x)=D_r(x)+s(B^n\cap H^{\perp})$, and note that $K\cap C_r(x)\in {{\mathcal{K}}^n_n}$.  From $K\cap C_r(x)\subset K$, the monotonicity of $\di$ gives $\di(K\cap C_r(x))\subset \di K$.
The projection invariance of $\di$ implies that
$$(\di(K\cap C_r(x)))|H=(K\cap C_r(x))|H\subset D_r(x)=C_r(x)|H$$
and hence, since $(\di K)|H^{\perp}\subset C_r(x)|H^{\perp}$, we have
\begin{equation}\label{cont}
\di(K\cap C_r(x))\subset (\di K)\cap C_r(x).
\end{equation}
With (\ref{cont}) in hand, the proof of Theorem~\ref{SteinerCompact}(i) may be followed from (\ref{dik}) onwards to conclude that (\ref{vni}) holds for ${\mathcal{H}}^i$-almost all $x\in \relint(K|H)$. By the projection invariance of $\di$ and continuity, this is enough to yield (\ref{vni}) for all $x\in H$.
\end{proof}

\begin{cor}\label{Steiner}
Let $H\in {\mathcal{G}}(n,n-1)$ and let $\di:{{\mathcal{K}}^n_n}\rightarrow {{\mathcal{K}}^n_n}_H$ be an $(n-1)$-symmetrization.  If $\di$ is monotonic, volume preserving, and either invariant on $H$-symmetric spherical cylinders or projection invariant, then $\di$ is Steiner symmetrization with respect to $H$.
\end{cor}

\begin{proof}
Let $K\in {{\mathcal{K}}^n_n}$.  If $\di$ is assumed to be projection invariant, then by Theorem~\ref{Schwarz} with $i=n-1$, we have
\begin{equation}\label{Glines}
V_1\left((\di K)\cap G\right)=V_1(K\cap G)
\end{equation}
for all lines $G$ orthogonal to $H$. Since $\di K$ is $H$-symmetric, this yields $\di K=S_HK$.  If $\di$ is assumed to be invariant on $H$-symmetric spherical cylinders, we can follow the proof of Theorem~\ref{SteinerCompact}(i) with $i=n-1$ to conclude that (\ref{Glines}) holds for ${\mathcal{H}}^{n-1}$-almost all lines $G$ orthogonal to $H$.  (As in the proof of Theorem~\ref{Schwarz}, it is necessary to make the restriction $x\in \relint(K|H)$ in order to ensure that $K\cap C_r(x)\in {{\mathcal{K}}^n_n}$.)  Since $\di K$ and $S_HK$ are convex bodies, (\ref{Glines}) holds for all lines $G$ orthogonal to $H$ by continuity and the conclusion follows as before.
\end{proof}

For maps $\di:{{\mathcal{K}}^n_n}\rightarrow {{\mathcal{K}}^n_n}_H$, Examples~\ref{Vex2} with $j=n$, \ref{ex2nww}, and \ref{ex3nww} with $j=n$ show that none of the assumptions can be omitted.

If in Corollary~\ref{Steiner}, it is only assumed that $\di$ does not decrease volume, i.e., $V_n(\di K)\ge V_n(K)$ for all $K\in {{\mathcal{K}}^n_n}$, instead of the condition that $\di$ is volume preserving, then the proofs of Theorem~\ref{SteinerCompact}(i) and Theorem~\ref{Schwarz} show that $S_HK\subset \di K$ for all $K\in {{\mathcal{K}}^n_n}$.  However, Example~\ref{ex2nww}(ii) shows that this is not enough to conclude that $\di$ is Steiner symmetrization.  Similarly, Example~\ref{ex2nww}(i) shows that it is not enough to assume that $\di$ does not increase volume instead of preserving volume.

For $j\in \{1,\dots,n-1\}$, the $j$th intrinsic volume does not increase under Steiner symmetrization; see \cite[p.~587]{Sch93}.  But again, Example~\ref{ex2nww}(ii) shows that in Corollary~\ref{Steiner}, the volume-preserving property of $\di$ cannot be replaced by the assumption that $\di$ does not increase the $j$th intrinsic volume for $j\in \{1,\dots,n-1\}$.

\begin{cor}\label{Steinercor}
Let $H\in {\mathcal{G}}(n,n-1)$ and let $\di:{\mathcal{K}}^n\rightarrow {\mathcal{K}}^n_H$ be an $(n-1)$-symmetrization.  If $\di$ is monotonic, volume preserving, and either invariant on $H$-symmetric spherical cylinders or projection invariant, then $\di K=S_HK$ for each $K\in {\mathcal{K}}^n$ not contained in a hyperplane orthogonal to $H$.
\end{cor}

\begin{proof}
Since $\di$ is volume preserving, we have $\di:{\mathcal{K}}^n_n\rightarrow {\mathcal{K}}^n_{nH}$ and hence by Corollary~\ref{Steiner}, the result holds when $K$ is a convex body.  Suppose that $\dim K<n$. For each $\ee>0$, $K+\ee B^n$ is a convex body, so $\di(K+\ee B^n)=S_H(K+\ee B^n)$.  Therefore, if $K$ is not contained in a hyperplane orthogonal to $H$, the monotonicity of $\di$ yields
$$\di K\subset \di(K+\ee B^n)=S_H(K+\ee B^n)\rightarrow K|H=S_HK$$
in the Hausdorff metric as $\ee\rightarrow 0+$.  If $\di$ is projection invariant, we obtain $\di K=S_HK$ immediately.  Otherwise, we know only that $\di K\subset K|H=S_HK$. Suppose that $(K|H)\setminus \di K\neq \emptyset$.  Choose $L\in {\mathcal{K}}^n$ with $L\subset K$ and $L|H\subset (K|H)\setminus \di K$. Since the previous argument shows that $\di L\subset L|H$, we obtain $\di L\not\subset \di K$, contradicting the monotonicity of $\di$.  Therefore $\di K= K|H=S_HK$ as before.
\end{proof}

\begin{ex}\label{ex4}
{\em Let $H\in {\mathcal{G}}(n,n-1)$ and for all $K\in {{\mathcal{K}}^n}$, let $\di K=S_HK$ if $K$ is not contained in a hyperplane orthogonal to $H$ and $\di K=K|H$ otherwise. Then $\di:{{\mathcal{K}}^n}\rightarrow {\mathcal{K}}^n_H$ is monotonic, volume preserving, idempotent, invariant on $H$-symmetric spherical cylinders, and projection invariant. However, $\di$ is not invariant on $H$-symmetric sets and is therefore not Steiner symmetrization; indeed, if $u\in S^{n-1}$ is orthogonal to $H$ and $K=[-au,au]$ for $a>0$, then $S_HK=[-au,au]\neq \{o\}=K|H=\di K$.
\qed}
\end{ex}

Returning to Theorem~\ref{Schwarz}, we note that the hypotheses stated there are not enough to conclude that $\di$ is Schwarz symmetrization.  This is shown by Example~\ref{ex6} with $j=n-i$ and $L\neq B^n\cap H^{\perp}$, as well as by the following different example, which together suggest that it may be difficult to find a nontrivial characterization of Schwarz symmetrization.

\begin{ex}\label{ex7}
{\em Let $i\in \{1,\dots,n-2\}$ and let $H\in {\mathcal{G}}(n,i)$.  Choose mutually orthogonal subspaces $H_j\in {\mathcal{G}}(n,n-1)$, $j=1,\dots,n-i$, such that $H=\cap_{j=1}^{n-i}H_j$.  For $K\in {{\mathcal{K}}^n_n}$, define $\di K=(S_{H_{1}}\circ S_{H_{2}}\circ\cdots\circ S_{H_{n-i}})K$.  Since $\di K$ is $H_j$-symmetric, $j=1,\dots,n-i$, it is also $H$-symmetric and hence $\di:{\mathcal{K}}^n_n\rightarrow {\mathcal{K}}^n_{nH}$ is an $i$-symmetrization.  Moreover, $\di$ is strictly monotonic, volume preserving, invariant on $H$-symmetric sets, and projection invariant.  However, $\di$ is not Schwarz symmetrization.
\qed}
\end{ex}

\section{Open problems}\label{problems}

In the problems below, we assume that $H\in {\mathcal{G}}(n,i)$, $i\in\{0,\dots,n-1\}$, and ${\mathcal{B}}={\mathcal{K}}^n$ or ${\mathcal{B}}={\mathcal{K}}^n_n$.

\begin{prob}\label{prob0}
Let $i=n-1$ and let $j\in \{2,\dots,n-1\}$.  Is there a symmetrization ${\di}:{\mathcal{B}}\rightarrow{\mathcal{B}}_H$ that is monotonic, $V_j$-preserving, and either invariant on $H$-symmetric spherical cylinders or projection invariant?
\end{prob}

In particular, taking $j=n-1$ in Problem~\ref{prob0}, is there a symmetrization on compact convex sets that behaves like Minkowski or Steiner symmetrization but which preserves surface area instead of mean width or volume?  Variants of Problem~\ref{prob0} may be posed, for example insisting that $\di$ be invariant on $H$-symmetric sets and extending the question to $i\in \{0,\dots,n-1\}$.

\begin{prob}\label{prob1}
Let $i\in \{2,\dots,n-2\}$.  Is there a symmetrization ${\di}:{\mathcal{B}}\rightarrow{\mathcal{B}}_H$ that is strictly monotonic, idempotent, and invariant on $H$-symmetric spherical cylinders, but not projection invariant?
\end{prob}

We remark that if such a $\di$ exists, it cannot be invariant under translations orthogonal to $H$ of $H$-symmetric sets.  Indeed, suppose it is and let $K\in {\mathcal{B}}$.  Then by Theorem~\ref{IdemGab2}, $\di K\subset O_HK$, and since $O_H$ is projection invariant, we obtain $(\di K)|H\subset (O_HK)|H= K|H$. The second paragraph of the proof of Theorem~\ref{Invariance} shows, assuming only monotonicity and invariance on $H$-symmetric spherical cylinders, that the reverse inclusion $K|H\subset (\di K)|H$ holds, so $\di$ is projection invariant.

\begin{prob}\label{prob2}
Let $i\in \{2,\dots,n-2\}$. Is there a strictly increasing set function $F:{\mathcal{B}}\to [0,\infty)$ and a symmetrization ${\di}:{\mathcal{B}}\rightarrow{\mathcal{B}}_H$ that is monotonic, $F$-preserving, and invariant on $H$-symmetric spherical cylinders, but not projection invariant?
\end{prob}

\begin{prob}\label{prob3}
Let $i\in \{0,\dots,n-2\}$.  Is there a symmetrization ${\di}:{\mathcal{B}}\rightarrow{\mathcal{B}}_H$ that is monotonic, mean width preserving, and invariant on $H$-symmetric sets, but not invariant under translations orthogonal to $H$ of $H$-symmetric sets?
\end{prob}

\end{document}